\documentclass[12pt]{amsart}

\usepackage[margin=1in]{geometry}  
\usepackage{graphicx}              
\usepackage{amsmath}               
\usepackage{amsfonts}              
\usepackage{amsthm}                
\usepackage{amsmath, amssymb}
\usepackage[all]{xy}
\usepackage{enumerate}
\usepackage{graphicx}
\usepackage{color}

\newtheorem{thm}{Theorem}[section]
\newtheorem{lem}[thm]{Lemma}
\newtheorem{prop}[thm]{Proposition}
\newtheorem{cor}[thm]{Corollary}

\newenvironment{remark}[1][Remark.]{\begin{trivlist}
\item[\hskip \labelsep {\bfseries #1}]}{\end{trivlist}}

\theoremstyle{definition}
\newtheorem{defn}[thm]{Definition}

\theoremstyle{definition}
\newtheorem{exmp}[thm]{Example}

\begin{document}

\nocite{*}

\title{Induced and Coinduced Modules over Cluster-Tilted Algebras}

\author{Ralf Schiffler}   
\address{Department of Mathematics, University of Connecticut, 
Storrs, CT 06269-3009, USA}
\email{schiffler@math.uconn.edu}
\author{Khrystyna Serhiyenko}\thanks{The authors were supported by the NSF CAREER grant DMS-1254567 and by the University of Connecticut.}
\address{Department of Mathematics, University of Connecticut, 
Storrs, CT 06269-3009, USA}
\email{khrystyna.nechyporenko@math.uconn.edu}

\maketitle

\begin{abstract}
We propose a new approach to study the relation between the module categories of a tilted algebra $C$ and the corresponding cluster-tilted algebra $B=C\ltimes E$. This new approach  consists of using  the induction functor $-\otimes_C B$ as well as the coinduction functor $D(B\otimes_C D-)$. 
 We show  that $DE$ is a partial tilting and a $\tau$-rigid $C$-module and that the induced module $DE\otimes_C B$ is a partial tilting and  a $\tau$-rigid $B$-module. Furthermore, if $C=\text{End}_A T$ for a tilting module $T$ over a hereditary algebra $A$, we compare the induction and coinduction functors to the Buan-Marsh-Reiten functor $\text{Hom}_{\mathcal{C}_A}(T,-)$ from the cluster-category of $A$  to the module category of $B$. We also study the question which $B$-modules are actually induced or coinduced from a module over a tilted algebra.
\end{abstract}

\section{Introduction}
Cluster-tilted algebras are finite dimensional associative algebras which were
introduced in \cite{BMR} and, independently, in \cite{CCS} for the type $\mathbb{A}$. 

One motivation for introducing these algebras came from Fomin and Zelevinsky's cluster algebras \cite{FZ}. To every cluster in an acyclic cluster algebra one can associate a cluster-tilted algebra, and the indecomposable rigid modules over the cluster-tilted algebra correspond bijectively to the cluster variables outside the chosen cluster. Generalizations of cluster-tilted algebras, the Jacobian algebras of quivers with potentials, were introduced in \cite{DWZ}, extending this correspondence to the non-acyclic types. 
Many people have studied cluster-tilted algebras in this context, see for example \cite{BBT,BMR, BMR2, BMR3, CCS2, CC, CK, KR}.

The second motivation came from classical tilting theory. Tilted algebras are the endomorphism algebras of tilting modules over hereditary algebras, whereas cluster-tilted algebras are the endomorphism algebras of cluster-tilting objects over  cluster categories of  hereditary algebras. This similarity in the two definitions lead to the following precise relation between tilted and cluster-tilted algebras, which was established in \cite{ABS}.

There is a surjective map
\[\xymatrix{ \{\textup{tilted algebras}\} \ar@{->>}[r] &\{\textup{cluster-tilted algebras}\} ,& C\ar@{|->}[r]&B=C\ltimes E,
}\]
where $E$ denotes the $C$-$C$-bimodule $E=\text{Ext}^2_C(DC,C)$ and $C\ltimes E$ is the trivial extension.

This result allows one to define cluster-tilted algebras without using the cluster category.
It is natural to ask how the module categories of $C$ and $B$ are related, and several results in this direction have been obtained, see for example \cite{ABS2,ABS3,ABS4, BFPPT,BOW,DS}. 

The Hochschild cohomology of the algebras $C$ and $B$ has been compared in \cite{AR,ARS,ABIS,L}.

\bigskip

In this paper, we use a new approach to study the relation between the module categories of a tilted algebra $C$ and its cluster-tilted algebra $B=C\ltimes E$, namely \emph{induction} and \emph{coinduction}.

The induction functor
$-\otimes_C B$ and the coinduction functor $\text{Hom}_C(B,-)$ from $\textup{mod}\,C$ to $\textup{mod}\,B$ are defined whenever $C$ is a subring of $B$ which has the same identity. If we are dealing with algebras over a field $k$, we can, and usually do, write the coinduction functor as $D(B\otimes_CD-)$, where $D=\text{Hom}(-,k)$ is the standard duality.

Induction and coinduction are important tools in classical Representation Theory of Finite Groups. In this case, $B$ would be the group algebra of a finite group $G$ and $C$ the group algebra of a subgroup of $G$ (over a field whose characteristic is not dividing the group orders). In this situation, the algebras are semi-simple, induction and coinduction are the same functor, and this functor is exact.

For arbitrary rings, and even for finite dimensional algebras, the situation is not that simple. In general, induction and coinduction are not the same functor and, since the $C$-module $B$ is not projective (and not flat), induction and coinduction are not exact functors.

However, the connection between tilted algebras and cluster-tilted algebras is close enough so that induction and coinduction are interesting tools for the study of the relation between the module categories. 

Induction and coinduction have been studied for split extension algebras in \cite{AM,AZ}, and we apply some of their results to our situation.

\smallskip

Our first main result is on the $C$-$C$-bimodule $E=\text{Ext}_C^2(DC,C)$ considering its
right $C$-module structure $E_C$ as well as its left $C$-module structure ${}_CE$ but as a right $C$-module $D({}_CE)$. We show the following.

\begin{thm}\label{thm 1}
If $C$ is a tilted algebra and $B$ is the corresponding cluster-tilted algebra, then \\
\indent{\upshape{(a)}}  $DE$ is a  (classical) partial tilting and $\tau_C$-rigid $C$-module, and its corresponding induced module $DE\otimes B$ is a (classical) partial tilting and $\tau_B$-rigid $B$-module.\\
\indent{\upshape{(b)}} $E$ is a partial cotilting and $\tau_C$-corigid $C$-module, and its corresponding coinduced module $D(B\otimes DE)$ is a partial cotilting and $\tau_B$-corigid $B$-module.  
\end{thm}
We think it would be an interesting problem to study the possible completions of these partial (co)-tilting modules and their endomorphism algebras.

\smallskip

Our second main result compares induction and coinduction with the well-known equivalence of categories
$\textup{Hom}_{\mathcal{C}_A}(T,-):\mathcal{C}_A/ \textup{add}\,\tau T \rightarrow \textup{mod}\,B$
of Buan, Marsh and Reiten \cite{BMR}.  Here $\mathcal{C}_A$ denotes the cluster category. Recall that mod$\,A$ naturally embeds in $\mathcal{C}_A$. We show that induction commutes with this equivalence on the subcategory $\mathcal{T}(T)$ given by the torsion subcategory determined by the tilting module $T$,
\[\mathcal{T}(T)=\{M\in\textup{mod}\,A\mid \textup{Ext}_A^1(T,M)=0\}.\]
We also show a similar statement for the coinduction functor using the torsion free subcategory $\mathcal{F}(T)$ determined by $T$,
\[\mathcal{F}(T)=\{M\in\textup{mod}\,A\mid \textup{Hom}_A(T,N)=0\}.\]
We show the following.

\begin{thm} \label{thm 3} Let $C$ be a tilted algebra and $B$ the corresponding cluster-tilted algebra. Then
\[\begin{array}{rcl}
 \textup{Hom}_A(T,M)\otimes_C B& \cong& \textup{Hom}_{\mathcal{C}_A}(T,M)$, for every $M\in \mathcal{T}(T),\\
D(B\otimes_C D\textup{Ext}_A^1(T,M))&\cong &\textup{Ext}_{\mathcal{C}_A}^1(T,M)$, for every $M\in \mathcal{F}(T).
\end{array}
\]
\end{thm}
Let us point out that these formulas implicitly use the tilting theorem of Brenner and Butler \cite{BB}.

\smallskip

Finally, we address the question which $B$-modules are actually induced or coinduced from a tilted algebra. We obtain the following results.
\begin{thm}\label{thm 4}
 Let $B$ be a cluster-tilted algebra.
 \begin{itemize}
\item[{\rm (a)}] If $B$ is of finite representation type, then \emph{every} $B$-module is induced and coinduced from some tilted algebra. 
\item[{\rm (b)}] If $B$ is of arbitrary representation type, then every \emph{transjective} $B$-module is induced or coinduced from some tilted algebra.
\item[{\rm (c)}] If $B$ is cluster concealed, then \emph{every} $B$-module is induced or coinduced from some tilted algebra.
\item[{\rm (d)}] If $B$ is of tame representation type and such that there are no morphisms between indecomposable projective non-transjective modules, then  \emph{every} indecomposable $B$-module is induced or coinduced from some tilted algebra.
\end{itemize}
\end{thm}
We remark that for general cluster-tilted algebras, there are indecomposable modules that are not induced and not coinduced.   Example \ref{ex 7.7} depicts a cluster-tilted algebra with infinitely many such modules, all lying in the same tube. We think it would be interesting to study the structure of these modules. 

\smallskip

Using induction and coinduction functors we also constructed explicit injective presentations for the induced modules in cluster-tilted algebras in \cite{SS}.

The paper is organized as follows. We start by recalling some background on tilting theory in section \ref{sect:basics} and reviewing general results on induction and coinduction in section~\ref{sect 3}.
In section \ref{sect 4}, we study induction and coinduction from tilted to cluster-tilted algebras and prove Theorem \ref{thm 1}. The relation to cluster categories and the proof of Theorem \ref{thm 3} are contained in section \ref{sect 6}.   In section \ref{sect 7}, we study the question which $B$-modules are induced or coinduced from a tilted algebra and prove Theorem \ref{thm 4}.  In particular, part (a) follows from Theorem \ref{7.0}, part (b) follows from Theorem \ref{7.1}, part (c) follows from Theorem \ref{7.2}, and pard (d) follows from part (b) and Theorem \ref{7.3}.

\section{Notation and Preliminaries}
\label{sect:basics}
Throughout this paper all algebras are assumed to be basic, finite dimensional over an algebraically closed field $k$.  Suppose $Q=(Q_0, Q_1)$ is a connected quiver without oriented cycles.  By $kQ$ we denote the path algebra of $Q$.  If $\Lambda$ is a $k$-algebra then denote by mod$\,\Lambda$  the category of finitely generated right $\Lambda$-modules and by ind$\,\Lambda$ a set of representatives of each isoclass of indecomposable right $\Lambda$-modules.  Given $M \in $ mod$\,\Lambda$, the projective dimension of $M$ in mod$\,\Lambda$ is denoted by pd$_\Lambda M$ and its injective dimension by id$_\Lambda M$. Let $\tau$ be the Auslander-Reiten translation.  Also define $\Omega M$ to be the first syzygy and $\Omega^{-1} M$ the first cosyzygy of $M$.  By $D$ we understand the standard duality functor Hom$_k(-,k)$.  Finally, let $\nu=D\text{Hom}_\Lambda(-,\Lambda)$ be the Nakayama functor and $\nu ^{-1}=\text{Hom}_\Lambda(D\Lambda,-)$ be the inverse Nakayama functor. For further details on representation theory we refer to \cite{ASS, S}. 

\subsection{Tilted algebras}
 We recall the classical definition of a tilting module.

\begin{defn}\label{def 1}
Let $\Lambda$ be an algebra.  A $\Lambda$-module $T$ is called a \emph{partial tilting module} if the following two conditions are satisfied:
\begin{itemize}
\item[-] pd$\,T_{\Lambda}\leq 1$,
\item[-] $\text{Ext}^1_{\Lambda}(T,T)=0$.
\end{itemize}
A partial tilting module $T$ is called a \emph{tilting module} if it also satisfied the following additional condition: 
\begin{itemize}
\item[-] There exists a short exact sequence $0\to \Lambda_{\Lambda}\to T'_{\Lambda} \to T''_{\Lambda}\to 0$ with $T', T'' \in \text{add}\,T$. 
\end{itemize}
\end{defn}

Let $A$ be a hereditary algebra and let $T$ be a basic tilting $A$-module.  It is well known that in this case, $T$ is tilting if Ext$_A^1 (T,T)=0$ and the number of indecomposable direct summands of $T$ equals the number of isomorphism classes of simple $A$-modules. The corresponding algebra $C=\text{End}_A T$ is called a \emph{tilted algebra}.  Consider the following full subcategories of $\textup{mod}\,A$.
$$ \mathcal{T}(T)=\{ M\in \text{mod}\,A \mid\text{Ext}_A^1(T,M)=0\}$$ 
$$ \mathcal{F}(T)=\{M\in \text{mod}\,A \mid\text{Hom}_A(T,M)=0\}$$
  Then $(\mathcal{T}(T), \mathcal{F}(T))$ is a torsion pair in mod$\,A$ that determines another torsion pair $(\mathcal{X}(T), \mathcal{Y}(T))$ in $\textup{mod}\,C$, where 
$$\mathcal{X}(T)=\{ N\in \text{mod}\,C\mid  N\otimes _C T =0\}$$
$$\mathcal{Y}(T)=\{N\in \text{mod}\,C \mid\text{Tor}_1^C (N, T)=0\}.$$
For more details refer to \cite[chapters VI, and VIII]{ASS}.  With this notation consider the following theorem due to Brenner and Butler, see \cite{BB} for details. 

\begin{thm} \label{2.1} Let $A$ be a hereditary algebra, $T$ a tilting $A$-module, and {\upshape{$C=\text{End}_A T$}}. Then \\
\indent {\upshape{(a)}}  the functors {\upshape{Hom}}$_A(T,-)$ and $-\otimes_C T$ induce quasi-inverse equivalences between $\mathcal{T}(T)$ and $\mathcal{Y}(T)$;\\
\indent {\upshape{(b)}} the functors {\upshape{Ext}}$_A^1(T, -)$ and {\upshape{Tor}}$_1^C (-,T)$ induce quasi-inverse equivalences between $\mathcal{F}(T)$ and $\mathcal{X}(T)$.  \\
Moreover, for any {\upshape $M\in\text{mod}\,A$} we have {\upshape $\text{Tor}_1^C(\text{Hom}_A(T,M),T)=0$} and {\upshape $\text{Ext}_A^1(T,M)\otimes _C T=0$}, and for any {\upshape $N\in\text{mod}\,C$} we have {\upshape $\text{Hom}_A(T,\text{Tor}_1^C (N,T))=0$} and {\upshape $\text{Ext}^1_A(T,N\otimes_C T)=0$}.   
\end{thm}

Let $\Lambda$ be a $k$-algebra, then let $\Gamma(\text{mod}\,\Lambda)$ denote the Auslander-Reiten quiver of mod$\,\Lambda$.  A connected component of $\Gamma(\text{mod}\,\Lambda)$ is called {\em regular} if it contains neither projective nor injective modules.  An indecomposable $\Lambda$-module is called {\em regular} if it belongs to a regular component of $\Gamma(\text{mod}\,\Lambda)$. 
There is a close relationship between the Auslander-Reiten quivers of a hereditary algebra and its   tilted algebras. We shall need the following result.

\begin{thm}\label{2.3}\cite[p.330]{ASS}
Let $A$ be a representation-infinite hereditary algebra, $T$ be a preprojective tilting $A$-module, and {\upshape $C=\text{End}_A T$}. Then\\
\indent{\upshape(a)} $\mathcal{T}(T)$ contains all but finitely many nonisomorphic indecomposable $A$-modules, and any indecomposable $A$-module not in $\mathcal{T}(T)$ is preprojective;\\
\indent{\upshape(b)} the image under the functor {\upshape $\text{Hom}_A(T,-)$} of regular $A$-modules yields all regular $C$-modules;\\
\indent{\upshape(c)} the injective and projective dimension of all regular $C$-modules is at most one.    
\end{thm}

The following proposition describes several facts about  tilted algebras, which we will use throughout the paper.  We recall some terminology first.

 \begin{defn}
Let $\Lambda$ be an algebra.  A {\em path} in mod$\,\Lambda$ with source $X$ and target $Y$ is a sequence of non-zero morphisms $X=X_0\rightarrow X_1\rightarrow \dots \rightarrow X_s=Y$ where $X_i\in\text{mod}\,\Lambda$ for all $i$, and $s\geq 1$.  In this case, $X$ is called a \emph{predecessor} of $Y$ in mod$\,\Lambda$ and $Y$ is called a \emph{successor} of $X$ in mod$\,\Lambda$.  

A {\em path} in $\Gamma(\text{mod}\,\Lambda)$ with source $X$ and target $Y$ is a sequence of arrows $X=X_0\rightarrow X_1\rightarrow \dots \rightarrow X_s=Y$ in the Auslander-Reiten quiver.  In addition,  such path is called {\em sectional} if for each $i$ with $0<i<s$, we have $\tau_\Lambda X_{i+1}\not = X_{i-1}$.  
\end{defn}

\begin{prop}  \label{2.2} Let $A$ be a hereditary algebra, $T$ a tilting $A$-module, and {\upshape{$C=\text{End}_A T$}} the corresponding tilted algebra. Then 
\begin{itemize}
\item[{\rm (a)}]  {\upshape{gl.dim}}C $\leq 2$.
\item[{\rm (b)}] For all {\upshape{$M\in \text{ind}\,C$}}  {\upshape{id}}$_C M \leq 1$ or  {\upshape{pd}}$_C M \leq 1$.
\item[{\rm (c)}] For all $M\in \mathcal{X}(T)$ {\upshape{id}}$_C M \leq 1$.
\item[{\rm (d)}] For all $M\in \mathcal{Y}(T)$ {\upshape{pd}}$_C M \leq 1$.
\item[{\rm (e)}] $(\mathcal{X}(T), \mathcal{Y}(T))$ is splitting, which means that every indecomposable C-module belongs either to $\mathcal{X}(T)$ or $\mathcal{Y}(T)$. 
\item[{\rm (f)}] $\mathcal{Y}(T)$ is closed under predecessors and $\mathcal{X}(T)$ is closed under successors. 
\end{itemize}
\end{prop}

There is also a precise description of injective modules in a tilted algebra in terms of the corresponding hereditary algebra.  

\begin{prop} \cite[Proposition VI 5.8]{ASS}\label{2.4}
Let $A$ be a hereditary algebra, $T$ a tilting $A$-module, and {\upshape $C=\text{End}_A T$}.  Let $T_1, \dots , T_n$ be a complete set of pairwise nonisomorphic indecomposable direct summands of $T$.  Assume that the modules $T_1, \dots ,T_m$ are projective, the remaining modules $T_{m+1},\dots , T_n$ are not projective, and $I_1, \dots , I_m$ are indecomposable injective $A$-modules with {\upshape $\text{soc}\, I_j\cong T_j/\text{rad}\, T_j$}, for $j=1,\dots , m$.  Then the $C$-modules 
{\upshape $$\text{Hom}_A(T, I_1),\dots ,\text{Hom}_A(T, I_m), \text{Ext}_A^1(T, \tau T_{m+1}),\dots ,\text{Ext}_A^1(T,\tau T_n)$$}
form a complete set of pairwise nonisomorphic indecomposable injective modules.
\end{prop}

\subsection{Cluster categories and cluster-tilted algebras}
Let $A=kQ$ and let  $\mathcal{D}=\mathcal{D}^b(\text{mod}\,A)$ denote the derived category of bounded complexes  of $A$-modules. The \emph{cluster category} $\mathcal{C}_A$ is defined as the orbit category of the derived category with respect to the functor $\tau_{\mathcal{D}}^{-1}[1]$, where $\tau_{\mathcal{D}}$ is the Auslander-Reiten translation in the derived category and $[1]$ is the shift.  Cluster categories were introduced in \cite{BMRRT}, and in \cite{CCS} for type $\mathbb{A}$, and were further studied in \cite{K,KR,A,P}.  They are triangulated categories \cite{K}, that  have Serre duality and are 2-Calabi Yau \cite{BMRRT}.  

An object $T$ in $\mathcal{C}_A$ is called \emph{cluster-tilting} if $\textup{Ext}^1_{\mathcal{C}_A}(T,T)=0$ and $T$ has $|Q_0|$ non-isomorphic indecomposable direct summands. The endomorphism algebra $\textup{End}_{\mathcal{C}_A}T$ of a cluster-tilting object is called a \emph{cluster-tilted algebra} \cite{BMR}.

The following theorem will be useful later.
\begin{thm}\label{2.5} \cite{BMR}
If $T$ is a cluster-tilting object in $\mathcal{C}_A$, then $\textup{Hom}_{\mathcal{C}_A}(T,-)$ induces an equivalence  of categories
{\upshape $\mathcal{C}_A/\text{add}(\tau T)\rightarrow \text{mod}\,\text{End}_{\mathcal{C}_A} T$}.
\end{thm}

\subsection{Relation extensions}
Let $C$ be an algebra of global dimension at most two and let $E$ be the $C$-$C$-bimodule $E=\text{Ext}_C ^2 (DC,C)$.
The \emph{relation extension} of $C$ is the trivial extension algebra $B=C\ltimes E$, whose underlying $C$-module is $C\oplus E$, and multiplication is given by $(c,e)(c',e')=(cc',ce'+ec')$.  
Relation extensions where introduced in \cite{ABS}. In the special case where $C$ is a tilted algebra, we have the following result.

\begin{thm}\cite{ABS}
Let $C$ be a tilted algebra. Then $B=C\ltimes \textup{Ext}_C ^2 (DC,C)$ is a cluster-tilted algebra. Moreover all cluster-tilted algebras are of this form.
\end{thm}

 \begin{remark}
This shows that the tilted algebra $C$ is a subalgebra and a quotient of the cluster-tilted algebra $B$.  Let $\pi: B\to C$ be the corresponding surjective algebra homomorphism.  We will often consider a $C$-module $M$ as a $B$-module with action $M\cdot b = M\cdot \pi(b)$.  In particular, $C$ and $E$ are right $B$-modules, and there exists a short exact sequence in mod$\,B$
\begin{equation}\label{(1)}    \xymatrix{0\ar[r]&E\ar[r]^i&B\ar[r]^{\pi}&C\ar[r]&0}. \end{equation}
\end{remark}

  \subsection{Slices and local slices}\label{sect 2.4}
Let $\Lambda$ be a $k$-algebra.

\begin{defn}
A {\em slice} $\Sigma$ in $\Gamma(\text{mod}\,\Lambda)$ is a set of indecomposable $\Lambda$-modules such that 
\begin{enumerate}[(i)]
\item $\Sigma$ is sincere,  meaning $\text{Hom}_{\Lambda}(P, \Sigma)\not=0$ for any projective $\Lambda$-module $P$. 
\item  Any path in mod$\,\Lambda$ with source and target in $\Sigma$ consists entirely of modules in $\Sigma$. 
\item  If $M$ is an indecomposable non-projective $\Lambda$-module then at most one of $M$, $\tau_\Lambda M$ belongs to $\Sigma$.
\item  If $M\rightarrow S$ is an irreducible morphism with $M,S\in\text{ind}\,\Lambda$ and $S\in\Sigma$, then either $M$ belongs to $\Sigma$ or $M$ is non-injective and $\tau^{-1}_\Lambda M$ belongs to $\Sigma$. 
\end{enumerate}
\end{defn}

Given an indecomposable $\Lambda$-module $Y$ in a connected component $\Gamma$ of $\Gamma(\text{mod}\,\Lambda)$ let 
$$\Sigma (\rightarrow Y) = \{ X\in\text{ind}\,\Lambda|\,\exists\; X\rightarrow \dots \rightarrow Y \in \Gamma \text{ and every path from } X \text{ to } Y\text{ in } \Gamma \text{  is sectional}\}$$
$$\Sigma (Y\rightarrow) = \{ X\in\text{ind}\,\Lambda|\,\exists\; Y\rightarrow \dots \rightarrow X \in \Gamma \text{ and every path from } Y \text{ to } X\text{ in } \Gamma \text{  is sectional}\}$$
be two subquivers of $\Gamma$. Consider the following result due to Ringel.

\begin{prop}\cite{R} \label{2.6}
 Let $Y$ be an indecomposable sincere module in a preprojective or preinjective component.  Then both $\Sigma(\rightarrow Y)$ and $\Sigma(Y\rightarrow)$ are slices. 
\end{prop}

The existence of slices is used to characterize tilted algebras in the following way. 

\begin{thm}\label{2.7} \cite{R}
Let {\upshape $C=\text{End}_A T$} be a tilted algebra.  Then  the $C$-module {\upshape $\text{Hom}_A(T,DA)$} forms a slice in {\upshape $\text{mod}\,C$}.  Conversely, any slice in any module category is obtained in this way. 
\end{thm}

The following notion of local slices has been introduced in \cite{ABS2} in the context of  cluster-tilted algebras. 

\begin{defn}
A {\em local slice} $\Sigma$ in $\Gamma(\text{mod}\,\Lambda)$ is a set of indecomposable $\Lambda$-modules inducing a connected full subquiver of $\Gamma(\text{mod}\,\Lambda)$ such that 
\begin{enumerate}[(i)]
\item If $X\in\Sigma$ and $X\rightarrow Y$ is an arrow in $\Gamma(\text{mod}\,\Lambda)$ then either $Y\in \Sigma$ or $\tau Y \in \Sigma$.
\item  If $Y\in \Sigma$ and $X\rightarrow Y$ is an arrow in $\Gamma(\text{mod}\,\Lambda)$ then either $X\in \Sigma$ or $\tau^{-1} X \in \Sigma$.
\item  For every sectional path $X=X_0\rightarrow X_1\rightarrow \dots \rightarrow X_s=Y$ in $\Gamma(\text{mod}\,\Lambda)$ with $X, Y \in \Sigma$ we have $X_i\in \Sigma$, for $i=0,1,\dots , s$. 
\item  The number of indecomposables in $\Sigma$ equals the number of nonisomorphic summands of $T$, where $T$ is a tilting $\Lambda$-module.  
\end{enumerate}
\end{defn}

\begin{remark}
The definition of a local slice makes sense if we replace the algebra $\Lambda$ by a cluster category $\mathcal{C}_A$, and then consider objects of $\mathcal{C}_A$ instead of $\Lambda$-modules. 
\end{remark}

There is a relationship between tilted and cluster-tilted algebras given in terms of slices and local slices.   Given a $\Lambda$-module $M$ let  $\text{Ann}_{\Lambda} M = \{a\in\Lambda\mid \, Ma = 0\}$ be its \emph{right annihilator}. 

\begin{thm}\label{2.8}\cite{ABS2}
Let $C$ be a tilted algebra and $B$ be the corresponding cluster-tilted algebra.  Then any slice in {\upshape $\text{mod}\,C$} embeds as a local slice in {\upshape $\text{mod}\,B$} and any local slice $\Sigma$ in {\upshape $\text{mod}\,B$} arises in this way.  Moreover, $C=B/\textup{Ann}_B \Sigma$. 
\end{thm}

The existence of local slices in a cluster-tilted algebra gives rise to the following definition.  The unique connected component of $\Gamma(\text{mod}\,B)$  that contains local slices is called the {\em transjective} component.  In particular if $B$ is of finite representation type then the transjective component is the entire Auslander-Reiten quiver of mod$\,B$. In this case there is the following statement. 

\begin{thm}\label{2.9}\cite{ABS2}
Let $B$ be a representation-finite cluster-tilted algebra.  Then any indecomposable $B$-module lies on a local slice. 
\end{thm}

We end this section with a lemma that we will need later.
\begin{lem}
 \label{lemls}
 Let $\widetilde \Sigma$ be a local slice in the cluster category $\mathcal{C}_A$. Then $\Sigma=\textup{Hom}_{\mathcal{C}_A}(T,\widetilde \Sigma)$ is a local slice in $\textup{mod}\,B$ if and only if $\widetilde \Sigma$ contains no summand of $\tau_{\mathcal{C}_A} T$.
\end{lem}
\begin{proof}
 $\text{Hom}_{\mathcal{C}_A}(T,-): \widetilde\Sigma\to\Sigma$ is a bijection if and only if $\widetilde \Sigma$ contains no summand of $\tau_{\mathcal{C}_A} T$.
Now the statement follows from \cite[Lemma 17]{ABS2}.
\end{proof}

\section{Induction and Coinduction Functors}\label{sect 3}

In this section we define two functors called induction and coinduction and describe some general results about them.  Suppose there are two $k$-algebras $C$ and $B$ with the property that $C$ is a subalgebra of $B$ and they share the same identity.  Then there is a general construction via the tensor product, also known as \emph{extension of scalars}, that sends a $C$-module to a particular $B$-module.  We give a precise definition below.
\begin{defn}
Let $C$ be a subalgebra of $B$, such that $1_C=1_B$, then 
$$-\otimes _C B:  \:  \text{mod}\,C \rightarrow \text{mod}\,B$$ 
is called the {\em{induction functor}}, and dually  
$$D(B\otimes_C D-): \: \text{mod}\,C \rightarrow \text{mod}\,B$$ 
is called the {\em{coinduction functor}}.   Moreover, given $M\in$ mod$\,C$ the corresponding {\em{induced module}} is defined to be $M\otimes _C B$, and the {\em{coinduced module}} is defined to be $D(B\otimes _C DM)$.  
\end{defn}
 
First observe that both functors are covariant.  The induction functor is right exact, while the coinduction functor is left exact.  Now consider the following lemma. 
\begin{lem} \label{*}
Let $C$ and $B$ be two $k$-algebras and $N$ a $C$-$B$-bimodule, then {\upshape{
$$M\otimes _ C N \cong D \text{Hom} _C ( M, DN)$$}}
 as $B$-modules for all {\upshape{$M \in \text{mod} \, C$}}. 
\end{lem}  
\begin{proof}
$M\otimes_C N \cong D \text{Hom}_k( M\otimes_C N, k)\cong D\text{Hom}_C(M, \text{Hom}_k (N, k))\cong D\text{Hom}_C (M, DN)$.  
\end{proof} 

The next proposition describes an alternative definition of these functors, and we will use these two descriptions interchangeably.
 
\begin{prop}\label{3.3}
Let $C$ be a subalgebra of $B$ such that $1_C=1_B$, then  for every $M\in$ {\upshape{mod}}$\,C$\\
\upshape{ \indent (a) $M\otimes_C B\cong D\text{Hom}_C(M,DB)$.\\
\indent (b) $D(B\otimes_C DM)\cong \text{Hom}_C(B, M)$.}
\end{prop}

\begin{proof}
Both parts follow from Lemma \ref{*}.
\end{proof}

Next we show some basic properties of these functors. 

\begin{prop} \label{3.4}
Let $C$ be a subalgebra of $B$ such that $1_C=1_B$.  If $e$ is an idempotent then\\
\indent {\upshape {(a)}} $(eC)\otimes _C B\cong eB$.\\
\indent{\upshape{(b)}} $D(B\otimes_C Ce)\cong DBe$. \\
In particular, if $P(i)$ and $I(i)$ are indecomposable projective and injective $C$-modules at vertex $i$, then $P(i)\otimes_C B$ and $D(B\otimes _C D I(i))$ are respectively indecomposable projective and injective $B$-modules at vertex $i$.  
\end{prop}
\begin{proof}
 Observe that $(eC)\otimes _C  B=e(C\otimes _C B)\cong eB$, and similarly we have $D(B\otimes_C (Ce))= D((B\otimes_C  C)e)\cong DBe$.  The rest of the proposition follows from above if we let $e=e_i$ be the primitive  idempotent given by the constant path at vertex $i$.  
\end{proof}

We can say more in the situation when $B$ is  a split extension of $C$.
  
\begin{defn}
Let $B$ and $C$ be two algebras.  We say {\em{$B$ is a split extension of $C$ by a nilpotent bimodule $E$}} if there exist algebra homomorphisms $\pi: B\to C$ and $\sigma: C \to B$ such that $\pi\sigma$ is the identity on $C$ and $E=\text{ker}\,\pi$ is a nilpotent (two-sided) ideal of $B$.   

In particular, there exists the following short exact sequence of $B$-modules. 
\begin{equation}\label{2}    \xymatrix{0\ar[r]&E\ar[r]^{i}&B \ar@<.5ex>[r]^{\pi}&C \ar@<.5ex>[l]^{\sigma}\ar[r]&0} \end{equation}
\end{defn}

\noindent For example, relation extensions are split extensions.

If $B$ is a split extension of $C$ then $\sigma$ is injective, which means $C$ is a subalgebra of $B$.  Also, $E$ is a $C$-$C$-bimodule, and we require $E$ to be nilpotent so that $1_B=1_C$.   Observe that $B\cong C\oplus E$ as $C$-modules, and there is an isomorphism of $C$-modules $M\otimes_C B\cong M\otimes _C C\oplus M\otimes _C E\cong M\oplus M\otimes_C E$.  Similarly, $D(B\otimes_C DM)\cong M\oplus D(E\otimes_C DM)$ as $C$-modules.  This shows that induction and coinduction of a module $M$ yields the same module $M$ plus possibly something else.  The next proposition shows a precise relationship between a given $C$-module and its image under the induction and coinduction functors. 

\begin{prop} \label{3.6}
Suppose $B$ is a split extension of $C$ by a nilpotent bimodule $E$, then for every  $M\in$ {\upshape{mod}}$\,C$ there exist two short exact sequences of $B$-modules:\\
\indent {\upshape{(a)}} \xymatrix{0\ar[r]&M\otimes_C E\ar[r]&M\otimes_C B\ar[r]& M\ar[r]&0}.\\
\indent {\upshape{(b)}} \xymatrix{0\ar[r]&M\ar[r]&D(B\otimes_C DM)\ar[r]& D(E\otimes_C DM)\ar[r]&0}.
\end{prop}
\begin{proof}
To show part (a) we apply $M\otimes_C -$ to the short exact sequence \eqref{2} and obtain the following long exact sequence 
$$\xymatrix{\text{Tor}^C_1(M,C)\ar[r]&M\otimes_C E\ar[r]& M\otimes _C B \ar[r]&M\otimes _C C\ar[r]&0}.$$ However, Tor$_1^C(M,C)=0$ since $C$ is a projective $C$-module, so part (a) follows. 
Similarly, to show part (b) we apply $D(-\otimes_C DM)$ to sequence \eqref{2}.  This yields a long exact sequence 
$$\xymatrix{0\ar[r]&D(C\otimes_C DM)\ar[r]&D(B\otimes_C DM)\ar[r]&D(E\otimes_C DM)\ar[r]&D\text{Tor}^C_1 (C,DM)}.$$
The last term in the sequence is again zero, which shows part (b). 
\end{proof}

Thus, in this situation each module is a quotient of its induced module and a submodule of its coinduced module.

\begin{prop} \label{3.7}
Suppose $B$ is a split extension of $C$ by a nilpotent bimodule $E$, then for every $M,N \in $ {\upshape{mod}}$\,C$ \\
 \indent{\upshape{(a)}} If $M \in \textup{ind}\,C$ then $M\otimes_C B$ and $D(B\otimes _C DM)$ lie in $\textup{ind}\,B$. \\
\indent{\upshape{(b)}}  $M\otimes _C B\cong N\otimes _C B$ if and only if $M\cong N$. \\
\indent{\upshape{(c)}}  $D(B\otimes _C  DM)\cong D(B\otimes _C  DN)$ if and only if $M\cong N$.
\end{prop}
\begin{proof}
Part (a).  Suppose $M \in \text{ind}\,C$, but $M \otimes_C B \cong M_1 \oplus M_2$ where neither $M_1$ nor $M_2$ is the zero module.  Because $B$ is a split extension of $C$, it follows that $C$ has a left $B$-module structure.  Thus, $M \otimes _C B \otimes _B C_C \cong (M_1 \oplus M_2) \otimes _B C_C$, and we conclude that $M \cong (M_1 \otimes _B C_C) \oplus (M_2 \otimes _B C_C)$.  By assumption $M$ is indecomposable, so without loss of generality let $M_1 \otimes _B C_C = 0$.  But $ 0 = M_1\otimes _B C \otimes_C B \cong M_1 \otimes _B B \cong M_1$, which is a contradiction.  Therefore, $M\otimes_C B$ is indecomposable.  A similar argument implies that if $M\in\text{ind}\,C$ then $D(B\otimes_C DM)\in\text{ind}\,B$.

Part (b).  Suppose $M\otimes_C B\cong N\otimes _C B$.  Again, since $C$ has a left $B$-module structure, we have that $M\otimes_C B \otimes _B C_C \cong N\otimes_C B \otimes _B C_C$, which means $M\cong M\otimes _C C_C \cong N\otimes_C C_C \cong N$.    The proof of part (c) is similar to that of part (b) and we omit it. 
\end{proof}

The next lemma describes a relationship between the Auslander-Reiten translations in $\textup{mod}\,C$ and $\textup{mod}\,B$, and induction and coinduction functors.  This lemma together with the following theorem were shown in \cite{AM}.  

\begin{lem} \label{3.8}
Suppose $B$ is a split extension of $C$ by a nilpotent bimodule, then for every $M \in $ {\upshape{mod}}$\,C$ \\
\indent{\upshape{(a)}}  $\tau_B (M\otimes_C B)\cong D(B\otimes_C D (\tau_C M))$.\\
\indent{\upshape{(b)}}  $\tau^{-1}_B D(B\otimes_C DM)\cong (\tau^{-1}_C M)\otimes_C B$.
\end{lem}

\begin{thm} \label{3.9}
Suppose $B$ is a split extension of $C$ by a nilpotent bimodule $E$ and $T\in$ {\upshape{mod}}$\,C$, then\\
\indent{\upshape{(a)}}  $T\otimes_C B$ is a (partial) tilting $B$-module if and only if $T$ is a (partial) tilting $C$-module, {\upshape{Hom}}$_C(T\otimes_C E, \tau_C T)=0$ and {\upshape{Hom}}$_C(DE,\tau_C T)=0$. \\
\indent{\upshape{(b)}} $D(B\otimes_C DT)$ is a (partial) cotilting $B$-module if and only if $T$ is a (partial) cotilting $C$-module, {\upshape{Hom}}$_C(\tau^{-1}_CT, D(E\otimes_C DT))=0$ and {\upshape{Hom}}$_C(\tau^{-1}_C T, E)=0$.
\end{thm}

\section{Induced and Coinduced Modules in  Cluster-Tilted Algebras}\label{sect 4}

In this section we develop properties of the induction and coinduction functors particularly when $C$ is an algebra of global dimension at most two and $B=C\ltimes E$ is the trivial extension of $C$ by the $C$-$C$-bimodule $E=\text{Ext}^2_C (DC,C)$. In the specific case when $C$ is also a tilted algebra, $B$ is the corresponding cluster-tilted algebra.  Some of the results in this section only hold when $C$ is tilted, but many hold in a more general situation when gl.dim$\,C\leq 2$, and we make that distinction clear in the assumptions of each statement.  However, throughout this section we always assume $E=\text{Ext}_C^2 (DC,C)$ and a tensor product $\otimes$ is a tensor product over $C$.  The main result of this section holds when $C$ is a tilted algebra.  It says that $DE$ is a partial tilting and $\tau _C$-rigid $C$-module, and the corresponding induced module $DE\otimes B$ is a partial tilting and $\tau _B$-rigid $B$-module.  

 We begin by establishing some properties of $E$ and $DE$ when the corresponding algebra $C$ has global dimension at most two.

\begin{prop} \label{4.1}
Let $C$ be an algebra of global dimension at most 2. Then \\
{\upshape{\indent (a) $E\cong \tau^{-1}\Omega^{-1} C $.\\
\indent (b) $DE\cong \tau\Omega DC$.\\
\indent (c) $M\otimes E\cong \tau^{-1}\Omega^{-1} M$.\\
\indent (d) $D(E\otimes DM)\cong \tau\Omega M$.}}
\end{prop}
\begin{proof}
Since the global dimension of $C$ is at most 2,  we have id$_C \Omega^{-1}M \leq 1$ (resp. pd$_C \Omega M \leq 1$) for all $C$-modules $M$. Therefore,  when applying the Auslander-Reiten formula below, we obtain the full Hom-space and not its quotient by the space of morphisms factoring through projectives (respectively injectives). \\
Part (a). $E=\text{Ext}^2_C(DC,C)\cong\text{Ext}_C^1(DC,\Omega^{-1}C)\cong D\text{Hom}_C(\tau^{-1}\Omega^{-1}C,DC)\cong \tau^{-1}\Omega^{-1}C$.\\
Part (b).  $DE=D\text{Ext}^2_C(DC,C)\cong D\text{Ext}_C^1(\Omega DC,C)\cong \text{Hom}_C(C,\tau\Omega DC)\cong \tau\Omega DC$.\\
Part (c).  Using Lemma \ref{*} we have $M\otimes E\cong D\text{Hom}_C(M,DE)$, which in turn by part (b) is isomorphic to $D\text{Hom}_C(M, \tau\Omega DC)$.  Then, we have the following chain of isomorphisms $M\otimes E\cong D\text{Hom}_C(M, \tau\Omega DC)\cong \text{Ext}^1_C(\Omega DC, M)\cong \text{Ext}^2_C (DC, M)\cong \text{Ext}^1_C( DC, \Omega^{-1} M)\cong D\text{Hom}_C(\tau^{-1}\Omega^{-1}M, DC)\cong \tau^{-1}\Omega^{-1} M$.  Part (d) can be shown in a similar manner as above.  
\end{proof}

\begin{prop}\label{4.3}
Let $C$ be an algebra of global dimension at most 2, and let $B=C\ltimes E$.  Suppose $M\in$ {\upshape{mod}}$\,C$,  then \\
\indent{\upshape{(a)}} {\upshape{id}}$_C M\leq 1$ if and only if $M\otimes B\cong M$.\\
\indent{\upshape{(b)}} {\upshape{pd}}$_C M\leq 1$ if and only if $D(B\otimes DM)\cong M$.
\end{prop}
\begin{proof}
Part (a). Recall that $M\otimes B\cong M\oplus M\otimes E$ as $C$-modules.  Therefore, it suffices to show that $M\otimes E=0$ if and only if id$_C M\leq 1$.    Proposition \ref{4.1}(c) implies that $M\otimes E\cong \tau^{-1}\Omega^{-1} M$, which is zero if and only if id$_C M\leq 1$. \\
Part (b).  Similarly it suffices to show that $D(E\otimes DM)=0$ if and only if pd$_C M\leq 1$.  However,  Proposition \ref{4.1}(d) implies that $D(E\otimes DM)\cong \tau\Omega M$, which again is zero if and only if pd$_C M\leq 1$.  
\end{proof}

\begin{lem} \label{4.2}
Let $C$ be an algebra of global dimension 2. Then  for all {\upshape{$M\in \text{mod}\,C$}}\\
\indent {\upshape{(a)}}  {\upshape{pd}}$_C N = 2$ for all nonzero {\upshape{$N\in\text{add}\, M\otimes E$}}. \\
\indent {\upshape{(b)}}  {\upshape{id}}$_C N = 2$ for all nonzero {\upshape{$N\in\text{add}\, D(E\otimes  DM)$}}.
\end{lem}
\begin{proof}
Part (a). By Proposition \ref{4.1} (c), $M\otimes E\cong \tau^{-1}\Omega^{-1}M$, which is nonzero if and only if id$_C M = 2$.   Now, consider a minimal injective resolution of $M$ and $\Omega^{-1}M$ 
$$\xymatrix@R-=.5cm {&&&\Omega^{-1}M\ar[dr]&&&&\\
0\ar[r]&M\ar[r]&I^0\ar[rr]\ar[ur]^{\pi}&&I^1\ar[r]&I^2\ar[r]&0.}$$
\noindent{Apply $\nu^{-1}$ to find a minimal projective resolution of $\tau^{-1}\Omega^{-1}M$.}
$$\xymatrix@R-=.5cm {
0\ar[r]&\nu^{-1}\Omega^{-1}M\ar[r]&\nu^{-1}I^1\ar[r]&\nu^{-1}I^2\ar[r]&\tau^{-1}\Omega^{-1}M\ar[r]&0.}$$
Since $C$ has global dimension two, $\nu^{-1}\Omega^{-1}M$ is a projective $C$-module.   It remains to show that it is nonzero.  By definition $\nu^{-1}\Omega^{-1}M=\text{Hom}_C(DC, \Omega^{-1}M)$, which is nonzero since we have a nonzero map $\pi$. Finally, observe that the argument above holds if we replace $\tau^{-1}\Omega^{-1}M$ by a nonzero direct summand of $\tau^{-1}\Omega^{-1}M$.  This completes the proof. Part (b) can be shown in a similar manner as above. 
\end{proof}


\begin{lem}\label{4.6}
Let $C$ be an algebra of global dimension at most 2,  then \\
\indent {\upshape{(a) Ext$_C^1(E,C)=0$.\\
\indent (b) Ext$_C^1(DC,DE)=0$.\\
\indent (c) Ext$_C^1 (E,E)=0$.\\
\indent (d) Ext$_C^1 (DE,DE)=0$.}}
\end{lem}
\begin{proof}
We will show parts (a) and (c), and the rest of the lemma can be proven similarly.  \\
Part (a).  By Proposition \ref{4.1}(a), we see that Ext$^1_C (E,C)\cong \text{Ext}^1_C(\tau^{-1}\Omega^{-1} C, C)$, which in turn by the Auslander-Reiten formula is isomorphic to $D\overline{\text{Hom}}_C (C,\Omega^{-1} C)$.  Let $i: C\rightarrow I$ be an injective envelope of $C$, thus we have the following short exact sequence 
\begin{equation}\label{(2)} \xymatrix{0\ar[r]&C\ar[r]^i &I\ar[r]^-{\pi}&\Omega^{-1}C\ar[r]&0}. \end{equation}
Applying Hom$_C(C,-)$ to this sequence we obtain an exact sequence
 $$\xymatrix{0\ar[r]& \text{Hom}_C(C,C)\ar[r]&\text{Hom}_C (C,I)\ar[r]^-{\pi_*}&\text{Hom}_C(C,\Omega^{-1}C)\ar[r]&\text{Ext}^1_C(C,C)}.$$ 
However,  Ext$^1_C(C,C)=0$ shows that $\pi_*$ is surjective.  This implies that every morphism from $C$ to $\Omega^{-1}C$ factors through the injective $I$. Thus,  $\overline{\text{Hom}}_C(C,\Omega^{-1}C)=0$, and this shows part (a). \\
Part (c).  As above observe that Ext$^1_C(E,E)\cong D \overline{\text{Hom}}_C(E,\Omega^{-1}C)$.  Applying Hom$_C(E,-)$ to the sequence \eqref{(2)} we get  an exact sequence
 $$\xymatrix{0\ar[r]& \text{Hom}_C(E,C)\ar[r]&\text{Hom}_C (E,I)\ar[r]^-{\pi_{**}}&\text{Hom}_C(E,\Omega^{-1}C)\ar[r]&\text{Ext}^1_C(E,C)}.$$ 
But then by part (a) we have Ext$^1_C(E,C)=0$, which shows that $\pi_{**}$ is surjective.  Thus $\overline{\text{Hom}}_C(E,\Omega^{-1}C)=0$, and this completes the proof of part (c).
\end{proof}   

\begin{cor}\label{4.7}
If the global dimension of $C$ is at most two, then both $E\oplus C$ and $DE\oplus DC$ are rigid modules. 
\end{cor}
\begin{proof}
Observe that $\text{Ext}^1_C (E\oplus C, E\oplus C)\cong \text{Ext}_C ^1 (E, E\oplus C) \oplus \text{Ext}^1_C (C, E\oplus C)$.  The first summand is zero because of Lemma \ref{4.6}, and the second is zero because $C$ is projective.  The proof of the rigidity of $DE\oplus DC$ is similar. 
\end{proof}

 Observe that, unless $C$ is hereditary, $E\oplus C$ is rigid but not tilting, because pd$_C E =2$, by Lemma \ref{4.2}. 
Next, we consider the case when $C$ is tilted.  

\begin{lem} \label{4.4}
Let $C$ be a tilted algebra. Then for all {\upshape{$M\in \text{mod} \,C$}} \\
\indent {\upshape{(a)}}  {\upshape{id}}$_C M\otimes E \leq 1$.\\
\indent {\upshape{(b)}}  {\upshape{pd}}$_C D(E\otimes DM) \leq 1$.
\end{lem}

\begin{proof}
Part (a) follows from Lemma \ref{4.2}(a) and Proposition \ref{2.2}(b).  Similarly, part (b) follows from Lemma \ref{4.2}(b) and Proposition \ref{2.2}(b).  
\end{proof}

Part (a) of the following proposition is well-known, see \cite{ABS, A}. 

\begin{prop}
Let $C$ be a tilted algebra. Then \\
\upshape{\indent (a) $E\otimes E=0$.\\
\indent (b) $D(E\otimes D(DE))=0$.}
\end{prop}
\begin{proof}
Part (a). Proposition \ref{4.1}(c) implies that $E\otimes E \cong \tau^{-1}\Omega^{-1}E$, but this is zero since id$_C E \leq 1$, by Lemma \ref{4.4}(a) with $M=C$.   Part (b) follows directly from part (a).
\end{proof}

\begin{remark}
The above proposition does not hold if $C$ has global dimension 2, but is not tilted.  For example, consider the algebra $C$ given by the following quiver with relations.
$$\xymatrix{1\ar[r]^{\alpha}&2\ar[r]^{\beta}&3\ar[r]^{\gamma}&4\ar[r]^{\delta}&5 && \alpha\beta=\gamma\delta=0.}$$
The Auslander-Reiten quiver of $C$ is the following.
\[\xymatrix@!@C=10pt@R=5pt{
&
{\begin{smallmatrix}4\\5\end{smallmatrix}}\ar[dr]&& &&&&{\begin{smallmatrix}1\\2\end{smallmatrix}}\ar[dr] \\
{\begin{smallmatrix}5\end{smallmatrix}}\ar[ur] &&
{\begin{smallmatrix}4\end{smallmatrix}}\ar[dr] &&
{\begin{smallmatrix}3\end{smallmatrix}}\ar[dr] &&
{\begin{smallmatrix}2\end{smallmatrix}}\ar[ur] &&
{\begin{smallmatrix}1\end{smallmatrix}}\\
&&&
{\begin{smallmatrix}3\\4\end{smallmatrix}}\ar[ur]\ar[dr] &&
{\begin{smallmatrix}2\\3\end{smallmatrix}}\ar[ur] &&
\\
&&&&
{\begin{smallmatrix}2\\3\\4\end{smallmatrix}}\ar[ur] &&
}\]
Here $E=S(1)\oplus S(3)$, where $S(i)$ is the simple module at vertex $i$.  Hence, $E\otimes E \cong \tau^{-1}\Omega^{-1} (S(1)\oplus S(3))=S(1)$.  
\end{remark}

This particular algebra $C$ also appears in \cite[Remark 4.21]{BFPPT}, where the authors stress that $C\ltimes E$ is not a cluster-tilted algebra.  On the other hand, the tensor algebra $T_C (E)$, which in this case equals $C \oplus E \oplus (E\otimes_C E)$, is a cluster-tilted algebra.

As a consequence of  Proposition \ref{4.3} and Lemma \ref{3.8} we obtain the following corollary, which says that a slice in a tilted algebra together with its $\tau$ and $\tau^{-1}$-translates fully embeds in the cluster-tilted algebra. This result was already shown in [ABS4] relying on the main theorem of [AZ]. Here we present a new proof using induction and coinduction functors. 

\begin{cor}\label{4.5}
Let $C$ be a tilted algebra and $B$ the corresponding cluster-tilted algebra.  Let $\Sigma$ be a slice in $\textup{mod}\,C$ and $M$ a module in $\Sigma$. Then \\
\indent{\upshape{(a)}} $\tau_C M\cong \tau_B M$,\\
\indent{\upshape{(b)}} $\tau^{-1}_CM\cong \tau_B^{-1} M$.
\end{cor}

\begin{proof}
We show part (a) and the proof of part (b) is similar.  Since the module $M$ lies on a slice in mod$\,C$ then pd$_C M\leq 1$ and id$_C M\leq 1$.  By Lemma \ref{3.8}(a) we have the following isomorphism 
$$\tau_B (M\otimes _C B)\cong D(B\otimes _C D(\tau _C M)).$$
It follows from Proposition \ref{4.3}(a) that the left hand side is isomorphic to $\tau_B M$.  It remains to show that the right hand side is isomorphic to $\tau_C M$. We may suppose without loss of generality that $C=\text{End}_A T$ and $\Sigma=\text{Hom}_A(T,DA)$ lies in $\mathcal{Y}(T)$.  Hence, in particular $M\in\mathcal{Y}(T)$.  Proposition \ref{2.2}(f) implies that $\tau_C M \in \mathcal{Y}(T)$, so part (d) of the same proposition shows that pd$_C \tau_C M \leq 1$.  Therefore, $D(B\otimes _C D (\tau_C M))\cong \tau_C M$ by Proposition \ref{4.3}(b), and this finishes the proof. 
\end{proof}

The following theorem is the main result of this section. Following \cite{AIR} we say that a $\Lambda$-module $M$ is $\tau_\Lambda$\emph{-rigid} if $\text{Hom}_\Lambda(M,\tau_\Lambda M)=0$,  and similarly that $M$ is $\tau_\Lambda$\emph{-corigid} if $\text{Hom}_{\Lambda}(\tau^{-1}_{\Lambda} M, M)=0$.  Also, $M$ is called \emph{partial cotilting} if id$_\Lambda \leq 1$ and $\text{Ext}^1_{\Lambda}(M,M)=0$. 

\begin{thm}\label{4.8}
If $C$ is a tilted algebra and $B$ is the corresponding cluster-tilted algebra, then \\
\indent{\upshape{(a)}}  $DE$ is a partial tilting and $\tau_C$-rigid $C$-module, and its corresponding induced module $DE\otimes B$ is a partial tilting and $\tau_B$-rigid $B$-module.\\
\indent{\upshape{(b)}} $E$ is a partial cotilting and $\tau_C$-corigid $C$-module, and its corresponding coinduced module $D(B\otimes DE)$ is a partial cotilting and $\tau_B$-corigid $B$-module. 
\end{thm}
\begin{proof}
We show part (a), and the proof of part (b) is similar.  First let us show that $DE$ is a partial tilting and $\tau_C$-rigid $C$-module.  Because $C$ is tilted, Lemma \ref{4.4}(b) implies that pd$_C DE \leq 1$, but also  $DE$ is rigid by Lemma \ref{4.6}(d).  This shows that $DE$ is a partial tilting $C$-module.  On the other hand,  by Lemma \ref{4.2}(b), all nonzero indecomposable summands of $DE$ have injective dimension 2.  Thus,  pd$_C DE\leq 1$, by Proposition \ref{2.2}(b).  So, applying the Auslander-Reiten formula we have $\text{Hom}_C(DE, \tau_C DE) \cong D\text{Ext}_C ^1 (DE, DE)=0$, where the last step follows from Lemma \ref{4.6}(d).  This shows that $DE$ is $\tau_C$-rigid.

Now, it remains to shows that  $DE\otimes B$ is a partial tilting and $\tau_B$-rigid $B$-module.  First we show that $DE\otimes B$ is partial tilting.  From above we know that $DE$ is a partial tilting $C$-module, thus by Theorem \ref{3.9} it suffices to show that Hom$_C(DE\otimes_C E, \tau_C DE)=0$ and Hom$_C(DE, \tau_C DE)=0$.  The second identity follows from the work above, so we need to show the first identity.  Observe that by Lemma \ref{4.2}(a), pd$_C DE\otimes E =2$ .  Then Proposition \ref{2.2} implies that $DE\otimes E \in \mathcal{X}(T)$.  Similarly, by Lemma \ref{4.2}(b),  id$_C DE =2$, so Proposition \ref{2.2} implies that $DE\in \mathcal{Y}(T)$.  However, by Proposition \ref{2.2}(f), $\mathcal{Y}(T)$ is closed under predecessors, which means $\tau_C DE \in \mathcal{Y}(T)$.  By definition of a torsion pair there are no nonzero morphisms from $\mathcal{X}(T)$ to $\mathcal{Y}(T)$, so Hom$_C(DE\otimes E, \tau_C DE)=0$.  This shows that $DE\otimes B$ is a partial tilting $B$-module.  

Now we show that $DE\otimes B$ is $\tau_B$-rigid, that is Hom$_B(DE\otimes B, \tau_B(DE\otimes B))=0$.  First observe that Lemma \ref{3.8}(a) yields $\tau_B(DE\otimes B)\cong D(B\otimes D\tau_C DE)$, which in turn by Proposition \ref{4.3}(b) is isomorphic to $\tau_C DE$.  Let $f\in \text{Hom}_B(DE\otimes B, \tau_C DE)$. Then we have the following diagram, whose top row is the short exact sequence of  Proposition \ref{3.6}(a).
$${\xymatrix {0\ar[r]&DE\otimes E \ar[r]^{i_*}&DE\otimes B \ar[r]^-{\pi_*}\ar[d]^{f}&DE\ar[r]\ar@{.>}[dl]^g&0\\
&&\tau_C DE&&}}. $$
Observe that $fi_*\in\text{Hom}_C(DE\otimes E, \tau_C DE)$, which is zero by our calculation above.  Next, the universal property of coker$\,i_*$ implies that there exists $g\in\text{Hom}_C(DE, \tau_C DE)$ such that $g\pi_*=f$.  However, we know that $DE$ is $\tau _C$-rigid, which  implies that $g=0$.  This means $f=0$.  Thus, we conclude that $DE\otimes B$ is $\tau_B$-rigid.  
\end{proof}

\begin{remark}
Unlike $DE$, the $C$-module $E$ is not $\tau_C$-rigid, that is Hom$_C(E,\tau_C E) \not=0$.  Consider, for example, the tilted algebra $C$  of type $\widetilde{\mathbb{A}}_{(2,2)}$ given by the following quiver with relations.
$${\xymatrix@R-=6cm {&2&\\1\ar[ur]^{\alpha}\ar[dr]_{\beta}&&4 \ar@<-.5ex>[ll]_{\delta} \ar@<.5ex>[ll]^{\gamma}&& \delta \alpha = \gamma\beta = 0.\\&3&}} $$
Here $E=\begin{smallmatrix} 4\\ 1\\2 \end{smallmatrix} \oplus \begin{smallmatrix} 4\\1\\3\end{smallmatrix} $ and $\tau _C E = \begin{smallmatrix} 4\\ 1 \end{smallmatrix} \oplus \begin{smallmatrix} 4\\1\end{smallmatrix} $.

Also, in this case $DE = (2 \oplus 3) \oplus 2 \oplus 3 \oplus (2\oplus 3)$, a direct sum of six simple projective $C$-modules. 
It is easy to see that $DE$ and $DE\otimes B$ are partial tilting and $\tau$-rigid in mod$\,C$ and mod$\,B$ respectively, because $DE$ is a projective $C$-module.
\end{remark}

\begin{exmp}

Let $C$ be the tilted algebra of type $\mathbb{D}_4$ given by the following quiver with relations.  

$${\xymatrix@R-=6cm {&2&\\1\ar[ur]^{\alpha}\ar[dr]_{\beta}&&4 \ar[ll]_{\delta} && \delta \alpha = \delta\beta = 0.\\&3&}} $$
Observe that $DE$ equals the indecomposable module  $\begin{smallmatrix}1\\ 2\;\;3 \end{smallmatrix}$.  On the other hand, $E$ equals $4\oplus 4\oplus 4$, a direct sum of three isomorphic simple modules.

\end{exmp}

\section{Relation to Cluster Categories}\label{sect 6}
In this section we study the relationship between induction and coinduction functors and the cluster category.  Here, we assume that $A$ is a hereditary algebra and $T\in$ mod$\,A$ is a basic tilting module.  Let $C=\text{End}_A T$ be the corresponding tilted algebra, and $B$ be the associated cluster-tilted algebra.   Finally, let $\mathcal{C}_A$ denote the cluster category of $A$.  We know that mod$\,A$ naturally embeds in $\mathcal{C}_A$, which in turn maps surjectively onto mod$\,B$ via the functor $\text{Hom}_{\mathcal{C}_A}(T,-)$.  Recall that the induction functor $-\otimes_C B$ and the coinduction functor $D(B\otimes_C D-)$ both map mod$\,C$ to mod$\,B$, while the module categories of $A$ and $C$ are closely related via the Tilting Theorem \ref{2.1}. Now, we want to study how the induction  and the coinduction functors fit into this larger picture. 
$$\xymatrix@C=3cm @R=1.2cm{\text{mod}\,A\ar@{<->}[r]^{\text{Tilting Theorem}}\ar[d]&\text{mod}\,C\ar[d]^{\begin{smallmatrix} -\otimes _C B\\\\D(B\otimes_C D-)\end{smallmatrix}}\\ \mathcal{C}_A\ar[r]^{\text{Hom}_{\mathcal{C}_A}(T,-)}&\text{mod}\,B}$$

The following theorem describes the relationship between the induction functor and the cluster category. 

\begin{thm}\label{9.1}
Let $A$ be a hereditary algebra and {\upshape $T\in \text{mod}\, A$} a basic tilting module.  Let $\mathcal{C}_A$ be the cluster category of $A$,  {\upshape $C=\text{End}_A T$} be the corresponding tilted algebra, and $B=C\ltimes E$  the corresponding cluster-tilted algebra.  Recall the definitions of $\mathcal{T}(T)$ and $\mathcal{Y}(T)$, the associated torsion class of {\upshape mod$\,A$} and the torsion free class of {\upshape mod$\,C$} below.
{\upshape $$\xymatrix{\mathcal{T}(T)=\{M\in\text{mod}\,A\mid \text{Ext}^1_A(T,M)=0\} & \mathcal{Y}(T)=\{N\in\text{mod}\,C\mid\text{Tor}^C_1(N,T)=0\}}$$}Then the following diagram commutes.  
{\upshape $$\xymatrix@C=3cm{\mathcal{T}(T)\ar[r]^{\text{Hom}_A(T,-)}\ar[d]&\mathcal{Y}(T)\ar[d]^{-\otimes _C B}\\ \mathcal{C}_A\ar[r]^{\text{Hom}_{\mathcal{C}_A}(T,-)}&\text{mod}\,B}$$}That is, {\upshape $\text{Hom}_A(T,M)\otimes_C B\cong \text{Hom}_{\mathcal{C}_A}(T, M)$} for every $M\in\mathcal{T}(T)$.  
\end{thm}

\begin{proof}
Let $M$ be an indecomposable module belonging to $\mathcal{T}(T)$.  If $M\in\text{add}\,T$, then $\text{Hom}_A(T,M)$ is a projective $C$-module, and $\text{Hom}_A(T,M)\otimes_C B$ is the corresponding projective $B$-module.  On the other hand, $\text{Hom}_{\mathcal{C}_A}(T,M)$ is also the same projective $B$-module.  Hence, in this case the theorem above holds.  If $M\not\in \text{add}\,T$, then Proposition \ref{2.2}(d) implies that the projective dimension of $\text{Hom}_A(T,M)\in\mathcal{Y}(T)$ is one.  Let 
\begin{equation} \label{12} \xymatrix{0\ar[r]&P_C^1\ar[r]^{f}&P_C^0\ar[r]&\text{Hom}_A(T,M)\ar[r]&0} \end{equation}
be its minimal projective resolution in mod$\,C$.  By Theorem \ref{2.1}(a), there exist $T^0, T^1\in\text{add}\,T$, and $g\in\text{Hom}_A(T^1,T^0)$ such that 
$$\xymatrix{f=\text{Hom}_A(T,g),&P^1_C=\text{Hom}_A(T,T^1),&P^0_C=\text{Hom}_A(T,T^0).}$$
Moreover, since $M\in\mathcal{T}(T)$, we have that $M=\text{coker}\,g$.  Applying $-\otimes_C T$ to the projective resolution (\ref{12}) and using Theorem \ref{2.1}(a), we obtain an exact sequence 
$$\xymatrix{\text{Tor}_1^C(\text{Hom}_A(T,M),T)\ar[r]&T^1\ar[r]^{g}&T^0\ar[r]&M\ar[r]&0.}$$
Furthermore, $\text{Tor}_1^C(\text{Hom}_A(T,M),T)=0$ by Tilting Theorem \ref{2.1}, which means that the sequence above is a short exact sequence in mod$\,A$.  Observe that $g$ is also a morphism in the cluster category, and the short exact sequence above corresponds to a triangle 
$$\xymatrix{T^1\ar[r]^g&T^0\ar[r]&M\ar[r]&T^1[1]}$$
in $\mathcal{C}_A$.  Applying $\text{Hom}_{\mathcal{C}_A}(T,-)$ to this triangle yields an exact sequence in mod$\,B$ 
$$\xymatrix{\text{Hom}_{\mathcal{C}_A}(T,T^1)\ar[r]^{g_*}&\text{Hom}_{\mathcal{C}_A}(T,T^0)\ar[r]&\text{Hom}_{\mathcal{C}_A}(T,M)\ar[r]&0.}$$
Note that there is zero on the right of the sequence as $\text{Hom}_{\mathcal{C}_A}(T,T^1[1])=0$, because $T$ is a tilting object in $\mathcal{C}_A$.  Also, $\text{Hom}_{\mathcal{C}_A}(T,T^1)=P_B^1$ and $\text{Hom}_{\mathcal{C}_A}(T,T^0)=P_B^0$,  thus 
$$\text{Hom}_{\mathcal{C}_A}(T,M)=\text{coker}\,\text{Hom}_{\mathcal{C}_A}(T,g):P_B^1\rightarrow P_B^0.$$  
On the other hand, applying the induction functor $-\otimes _C B$ to the sequence (\ref{12}) we obtain 
$$\xymatrix{P^1_B\ar[r]^{f\otimes 1}&P^0_B\ar[r]&\text{Hom}_A(T,M)\otimes_C B\ar[r]&0.}$$
Hence, we see that 
$$\text{Hom}_A(T,M)\otimes_C B = \text{coker}\,\text{Hom}_A(T,g)\otimes 1:P^1_B\rightarrow P_B^0.$$
But since $\text{top}\,P^1_B=\text{top}\,P^1_C$, we have 
$$\text{Hom}_{\mathcal{C}_A}(T,g)(\text{top}\,P^1_B)=\text{Hom}_A(T,g)(\text{top}\,P^1_C)=\text{Hom}_A(T,g)\otimes 1\, (\text{top}\,P^1_B).$$  
Because $P^1_B$ is projective both maps $\text{Hom}_{\mathcal{C}_A}(T,g)$ and $\text{Hom}_A(T,g)\otimes 1$ are determined by their restrictions to $\text{top}\,P^1_B$. Therefore, we construct a commutative diagram 
$$\xymatrix@C=2.5cm{P^1_B\ar[r]^{\text{Hom}_{\mathcal{C}_A}(T,g)}\ar@{=}[d]&P^0_B\ar[r]\ar[d]^{\cong}&\text{Hom}_{\mathcal{C}_A}(T,M)\ar[r]\ar[d]^{h}&0\\
P^1_B\ar[r]^{\text{Hom}_A(T,g)\otimes 1}&P^0_B\ar[r]&\text{Hom}_A(T,M)\otimes_C B\ar[r]&0}$$
where $h$ is an isomorphism by the Five Lemma.  This shows that for every $M\in \mathcal{T}(T)$ we have ${\text{Hom}_{\mathcal{C}_A}(T,M)\cong \text{Hom}_A(T,M)\otimes_C B}$, and finishes the proof of the theorem.
\end{proof}

This result shows that the induction functor can be reformulated in terms of other well-studied functors.  In particular consider the following corollary.  Recall the definitions of $\mathcal{F}(T)$ and $\mathcal{X}(T)$, the associated torsion free class of {\upshape mod$\,A$} and the torsion class of {\upshape mod$\,C$} below.
{\upshape $$\xymatrix{\mathcal{F}(T)=\{M\in\text{mod}\,A\mid \text{Hom}_A(T,M)=0\} & \mathcal{X}(T)=\{N\in\text{mod}\,C\mid N\otimes_C T=0\}}$$}

\begin{cor} \label{9.5}
Let $A$ be a hereditary algebra and {\upshape $T\in \text{mod}\, A$} a basic tilting module.  Let $\mathcal{C}_A$ be the cluster category of $A$,  {\upshape $C=\text{End}_A T$} be the corresponding tilted algebra, and $B=C\ltimes E$  the corresponding cluster-tilted algebra. Then
{\upshape \begin{displaymath}
   M\otimes_C B \cong \left\{
     \begin{array}{lrr}
       \text{Hom}_{\mathcal{C}_A}(T,M\otimes_C T) & \text{if}&M\in\mathcal{Y}(T)\\
       M & \text{if} & M\in\mathcal{X}(T)
     \end{array}
   \right.
\end{displaymath} }
for every {\upshape $M\in\text{ind}\,C$}. 
\end{cor}

\begin{proof}
First, note that every $M\in\text{ind}\,C$ belongs to either $\mathcal{X}(T)$ or $\mathcal{Y}(T)$ by Proposition~\ref{2.2}(e).  If  $M\in\mathcal{X}(T)$, Proposition \ref{2.2}(c) implies that id$_C M \leq 1$, and then Proposition \ref{4.3}(a) yields $M\otimes _C B \cong M$.  If $M\in\mathcal{Y}(T)$, then Theorem \ref{2.1}(a) implies that $M\cong\text{Hom}_A(T,N)$ for $N=M\otimes_C T\in \mathcal{T}(T)$.  Hence 
$M\otimes_C B \cong \text{Hom}_{\mathcal{C}_A}(T,N)$, by Theorem \ref{9.1}.
\end{proof}

Similarly, there is a dual statement that provides an alternative description of the coinduction functor. 

\begin{thm}\label{9.2}
Let $A$ be a hereditary algebra and {\upshape $T\in \text{mod}\, A$} a basic tilting module.  Let $\mathcal{C}_A$ be the cluster category of $A$,  {\upshape $C=\text{End}_A T$} be the corresponding tilted algebra, and $B=C\ltimes E$ the corresponding cluster-tilted algebra.  Then the following diagram commutes.  
{\upshape $$\xymatrix@C=3cm{\mathcal{F}(T)\ar[r]^{\text{Ext}^1_A(T,-)}\ar[d]&\mathcal{X}(T)\ar[d]^{D(B\otimes _C D-)}\\ \mathcal{C}_A\ar[r]^{\text{Ext}^1_{\mathcal{C}_A}(T,-)}&\text{mod}\,B}$$}That is, {\upshape $D(B\otimes_C D \text{Ext}_A^1(T,M))\cong\text{Ext}_{\mathcal{C}_A}^1(T,M)$}, for every $M\in\mathcal{F}(T)$.  
\end{thm}

\begin{proof}
Let $M$ be an indecomposable module belonging to $\mathcal{F}(T)$.  Consider $\text{Ext}_A^1(T,M)$, which belongs to $\mathcal{X}(T)$ by Theorem \ref{2.1}(b). If $\text{Ext}_A^1(T,M)=I^i_C$ is an  injective $C$-module, then $D(B\otimes_C DI^i_C)\cong I^i_B$ is the corresponding injective in mod$\,B$.  On the other hand, by Proposition \ref{2.4} we have $M\cong \tau T^i$ for some $T^i\in\text{add}\,T$.  Therefore,  by Serre Duality  
\begin{equation}
\label{14} \text{Ext}_{\mathcal{C}_A}^1(T, \tau T^i)\cong D\text{Hom}_{\mathcal{C}_A}(T^i,T)\cong I^i_B.
\end{equation}
  This shows that the theorem holds if  $\text{Ext}_A^1(T,M)$ is injective.  

If $\text{Ext}_A^1(T,M)\in\mathcal{X}(T)$ is not injective then according to Proposition \ref{2.2}(c) it has injective dimension one.  Consider a minimal injective resolution of this module in mod$\,C$ below. 
\begin{equation} \label{13} \xymatrix{0\ar[r]&\text{Ext}_A^1(T,M)\ar[r]&I^0_C\ar[r]^f&I^1_C\ar[r]&0}\end{equation}
Because $I^0_C, I^1_C$ are successors of $\text{Ext}_A^1(T,M)$, Proposition \ref{2.2}(f) implies that these injectives belong to $\mathcal{X}(T)$.  By Proposition \ref{2.4} and Theorem \ref{2.1}(b), there exist $T^1,T^0\in\text{add}\,T$, and $g\in \text{Ext}_A^1(\tau T^0,\tau T^1)$ such that 
$$\xymatrix{f=\text{Ext}_A^1(T,g),&I^1_C=\text{Ext}_A^1(T,\tau T^1),&I^0_C=\text{Ext}^1_A(T,\tau T^0).}$$
Applying $\text{Tor}^C_1(-,T)$ to sequence (\ref{13}) and using Theorem \ref{2.1}, we obtain a long exact sequence 
$$\xymatrix{\text{Tor}^C_2(I^1_C,T)\ar[r]&M\ar[r]&\tau T^0\ar[r]^g&\tau T^1\ar[r] & \text{Ext}_A^1(T,M)\otimes_C T.}$$
Observe that by the same theorem the last term in the sequence above is zero, likewise $\text{Tor}^C_2(I^1_C,T)=0$, because $T$ as a left $C$-module has projective dimension at most one.  Thus we obtain a short exact sequence in mod$\,A$ 
$$\xymatrix{0\ar[r]&M\ar[r]&\tau T^0\ar[r]^g&\tau T^1 \ar[r]&0}$$
which induces a triangle 
$$\xymatrix{\tau T^1[-1]\ar[r]&M\ar[r]&\tau T^0 \ar[r]^g&\tau T^1}$$
in the cluster category.
Applying $\text{Ext}^1_{\mathcal{C}_A}(T,-)$ to this triangle we obtain an exact sequence in mod$\,B$
$$\xymatrix{\text{Ext}^1_{\mathcal{C}_A}(T,\tau T^1[-1])\ar[r]&\text{Ext}^1_{\mathcal{C}_A}(T,M)\ar[r]&\text{Ext}^1_{\mathcal{C}_A}(T,\tau T^0)\ar[r]^{g_*}&\text{Ext}^1_{\mathcal{C}_A}(T,\tau T^1)}.$$
Observe that because $T$ is a tilting object in $\mathcal{C}_A$, we have $\text{Ext}^1_{\mathcal{C}_A}(T,\tau T^1[-1])=0$.  Moreover, equation (\ref{14}) yields $\text{Ext}^1_{\mathcal{C}_A}(T,\tau T^0)=I^0_B$ and $\text{Ext}^1_{\mathcal{C}_A}(T,\tau T^1)=I^1_B$.  Thus,
$$\text{Ext}^1_{\mathcal{C}_A}(T,M) = \text{ker}\, \text{Ext}^1_{\mathcal{C}_A}(T,g): I^0_B\rightarrow I^1_B.$$
On the other hand, applying the coinduction functor to sequence (\ref{13}) we obtain 
$$\xymatrix@C=2cm{0\ar[r]&D(B\otimes _C D \text{Ext}^1_A(T,M))\ar[r]&I^0_B\ar[r]^{D(1\otimes D g)}&I^1_B.}$$
Hence, we see that 
$$D(B\otimes _C D \text{Ext}^1_A(T,M))=\text{ker} \,D(1\otimes Dg) : I^0_B\rightarrow I^1_B.$$
Because $I^0_B$ is injective, both maps $\text{Ext}^1_{\mathcal{C}_A}(T,g)$ and $D(1\otimes Dg)$ are determined by their preimages of $\text{soc}\,I^1_B$.  But since $\text{soc}\,I^1_B=\text{soc}\,I^1_C$, we have 
$$\text{Ext}^1_{\mathcal{C}_A}(T,g)^{-1}(\text{soc}\, I^1_B)=\text{Ext}^1_A(T,g)^{-1}(\text{soc}\,I^1_C)=D(B\otimes _C D\text{Ext}^1_A(T,g))^{-1}(\text{soc}\, I^1_B).$$
Thus we have a commutative diagram 
$$\xymatrix@C=3cm{0\ar[r]&\text{Ext}^1_{\mathcal{C}_A}(T,M)\ar[r]\ar[d]^h&I^0_B\ar[d]^{\cong}\ar[r]^{\text{Ext}^1_{\mathcal{C}_A}(T,g)}&I^1_B\ar@{=}[d]\\
0\ar[r]&D(B\otimes_C D\text{Ext}^1_A(T,M))\ar[r]&I^0_B\ar[r]^{D(B\otimes_C D\text{Ext}^1_A(T,g))}&I^1_B}$$
and $h$ is an isomorphism by the Five Lemma. This shows that for every $M\in\mathcal{F}(T)$ we have $\text{Ext}^1_{\mathcal{C}_A}(T,M)\cong D(B\otimes_C D\text{Ext}_A^1(T,M))$, and finishes the proof of the theorem.  
\end{proof}

\begin{cor}
Let $A$ be a hereditary algebra and {\upshape $T\in \text{mod}\, A$} a basic tilting module.  Let $\mathcal{C}_A$ be the cluster category of $A$,  {\upshape $C=\text{End}_A T$} be the corresponding tilted algebra, and $B=C\ltimes E$  the corresponding cluster-tilted algebra. Then
{\upshape \begin{displaymath}
   D(B\otimes_C DM) \cong \left\{
     \begin{array}{lrr}
       \text{Ext}^1_{\mathcal{C}_A}(T,\text{Tor}^C_1(M,T)) & \text{if}&M\in\mathcal{X}(T)\\
       M & \text{if} & M\in\mathcal{Y}(T)
     \end{array}
   \right.
\end{displaymath} }
for every {\upshape $M\in \text{ind}\,C$}.  
\end{cor}
\begin{proof}
The proof is similar to that of Corollary 6.2 and we omit it. 
\end{proof}

\begin{exmp}
Let $A$ be the path algebra of  type $\mathbb{D}_6$ given by the following quiver. 
$$\xymatrix@R=.2cm{&&&&5\ar[dl]\\1&2\ar[l]&3\ar[l]&4\ar[l]\ar[dr]\\&&&&6}$$
Let 
$$T={\begin{smallmatrix}1\end{smallmatrix}}\oplus{\begin{smallmatrix}5\\4\\6\;3\\\;\;\;2\\\;\;\;1\end{smallmatrix}}\oplus
{\begin{smallmatrix}5\\4\\3\\2\\1\end{smallmatrix}}
\oplus{\begin{smallmatrix}5\\4\\6\;3\end{smallmatrix}}\oplus
{\begin{smallmatrix}5\\4\\6\end{smallmatrix}}\oplus{\begin{smallmatrix}5\end{smallmatrix}}$$
be the tilting $A$-module.  The Auslander-Reiten quiver of mod$\,A$ is shown at the top of Figure~\ref{figure 1}. The modules belonging to $\mathcal{T}(T)$ are in the dark shaded regions enclosed by a solid line, while modules belonging to $\mathcal{F}(T)$ are in the light shaded regions enclosed by a dotted line.

\begin{figure}
\begin{center}
{ \scalebox{0.9}{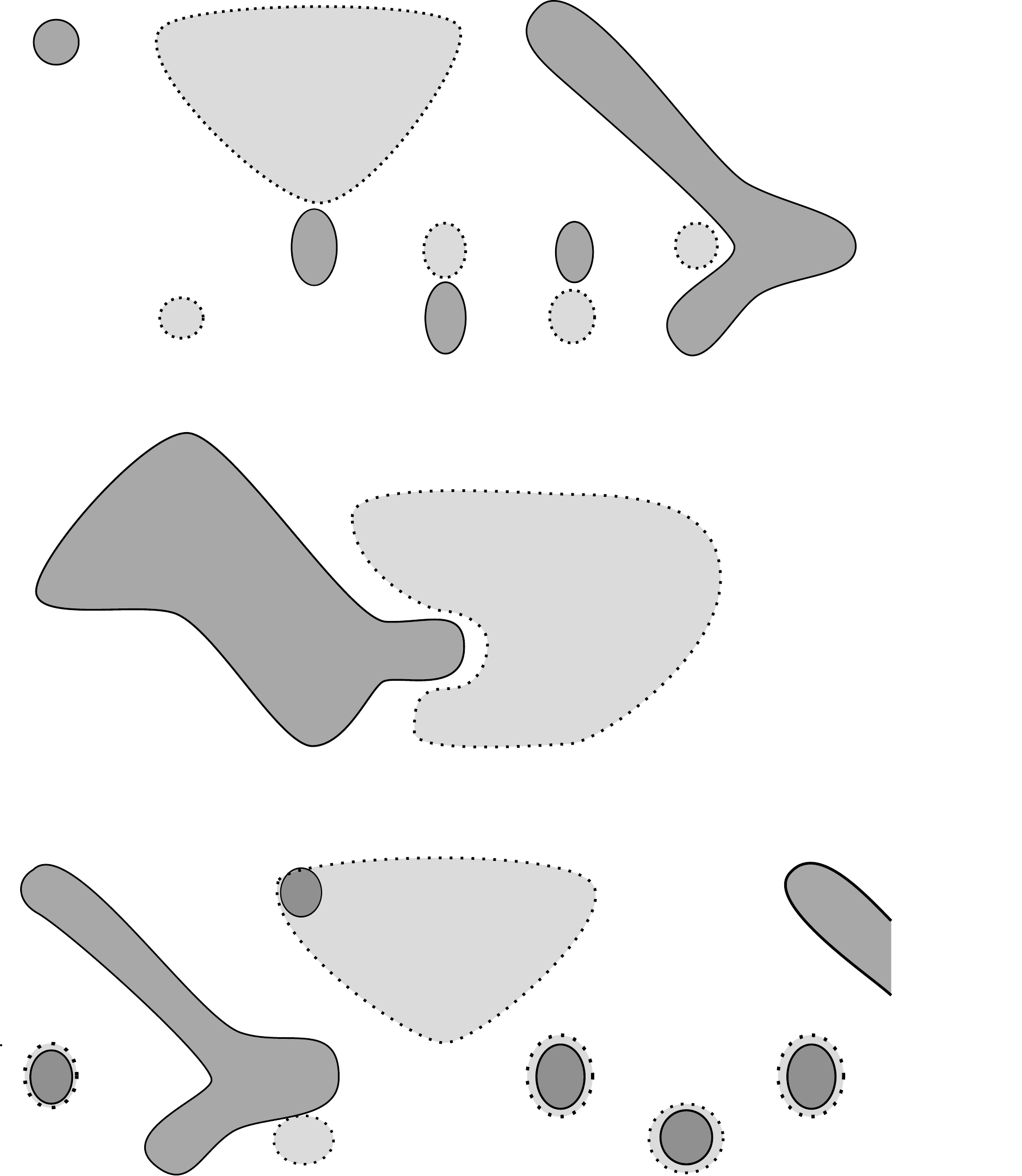}}
\caption{Auslander-Reiten quivers of $\textup{mod}\,A$ (top), $\textup{mod}\,C$ (middle) and $\textup{mod}\,B$ (bottom). }
\label{figure 1}
\end{center}
\end{figure}

Then the corresponding tilted algebra $C=\text{End}_A T$ is  of type $\mathbb{D}_6$, and it is given by the following quiver with relations.
$$\xymatrix@R=10pt{1&2\ar[l]_{\beta}&4\ar[l]_{\alpha}&5\ar[l]_{\gamma}&6\ar[l]_{\delta}\ar[dl]^{\epsilon}&&\alpha\beta=0\\&&&3\ar[ull]^{\sigma}&&&\delta\gamma\alpha=\epsilon\sigma.}$$
\end{exmp}

The Auslander-Reiten quiver of mod$\,C$ is shown in the middle of Figure \ref{figure 1}.  The modules belonging to $\mathcal{Y}(T)$ are in the dark shaded regions  enclosed by a solid line, while the modules belonging to $\mathcal{X}(T)$ are in the light shaded regions enclosed by a dotted line.

The cluster-tilted algebra $B=C\ltimes E$ is represented by the following quiver with relations. 

$$\xymatrix{1\ar@/^15pt/[rr]^{\omega}&2\ar[l]^{\beta}\ar@/^-15pt/[rrr]_{\;\;\;\rho}&4\ar[l]_{\alpha}&5\ar[l]_{\gamma}&6\ar[l]_{\delta}\ar[dl]^{\epsilon}&&{\begin{matrix}\alpha\beta=\omega\alpha=0& \gamma\alpha\rho=\alpha\rho\delta=0&\sigma\rho=\rho\epsilon=0\end{matrix}}\\&&&3\ar[ull]^{\sigma}&&&{\begin{matrix}\beta\omega=\rho\delta\gamma& \delta\gamma\alpha=\epsilon\sigma.\end{matrix}}}$$

The Auslander-Reiten quiver of mod$\,B$ is shown at the bottom of Figure \ref{figure 1}, where we identify the modules which have the same labels.  The modules  in the dark shaded regions enclosed by a solid line correspond to the set $\{Y\otimes_C B\mid Y\in\mathcal{Y}(T)\}$, while those in the light shaded regions enclosed by a dotted line correspond to the set $\{D(B\otimes_C DX)\mid X\in\mathcal{X}(T)\}$.  Note that there are four modules that lie in both sets. As stated in Theorems \ref{9.1} and \ref{9.2} the shape and relative position of $\mathcal{T}(T)$ and $\mathcal{F}(T)$ in mod$\,A$ correspond exactly to the set of induced $\mathcal{Y}(T)$-modules and the set of coinduced $\mathcal{X}(T)$-modules respectively. Moreover, if we apply $\tau^{-1}_B$ to the set of modules $\{D(B\otimes_C DX)\mid X\in\mathcal{X}(T)\}$, then their position relative to the induced $\mathcal{Y}(T)$ modules will correspond exactly to the position of $\mathcal{F}(T)$ relative to $\mathcal{T}(T)$ in mod$\,A$.

\section{Which Modules are Induced or Coinduced} \label{sect 7}

If we want to study the module category of a cluster-tilted algebra via the induction and coinduction functors,  then it is natural to ask   which modules in a cluster-tilted algebra are actually induced or coinduced from the modules over some tilted algebra. Recall from Theorem~\ref{2.8} that, given a cluster-tilted algebra $B$, every local slice $\Sigma$ gives rise to a tilted algebra $C=B/\text{Ann}\,\Sigma$, whose relation extension is $B$. 
\begin{defn}
 Let $B$ be a cluster-tilted algebra and $M$ a $B$-module.
\begin{enumerate}
\item $M$ is said to be \emph{induced from some tilted algebra} if there exists a tilted algebra $C$ and a $C$-module $X$ such that $B$ is the relation extension of $C$ and $M=X\otimes_C B$.
\item  $M$ is said to be \emph{coinduced from some tilted algebra} if there exists a tilted algebra $C$ and a $C$-module $X$ such that $B$ is the relation extension of $C$ and $M=D(B\otimes_C D X)$. 
\end{enumerate}
\end{defn}
We begin by looking at algebras of finite representation type.  

\begin{thm}\label{7.0}  If $B$ is a cluster-tilted algebra of finite representation type, then for every indecomposable $B$-module $M$ there exists a tilted algebra $C$ and a $C$-module $N$ such that $M$ is induced and coinduced from $N$.    
\end{thm}

\begin{proof}
If $B$ is of finite representation type and $M\in\text{ind}\,B$ then by Theorem \ref{2.9} there exists a local slice $\Sigma\in\Gamma(\text{mod}\,B)$ containing $M$.  Theorem \ref{2.8} implies that there is a tilted algebra $C$ such that $\Sigma$ is a slice in $\Gamma(\text{mod}\,C)$ and $B$ is the relation extension of $C$.  Hence, $M$ is an indecomposable $C$-module lying on a slice, which means id$_C M \leq 1$ and pd$_C M\leq 1$.  It follows from Proposition \ref{4.3} that $M\cong M\otimes _C B \cong D(B\otimes _C DM)$.  Therefore, $M\in\text{ind}\,B$ is both induced and coinduced from the same module $M\in\text{ind}\,C$, for some tilted algebra $C$.  
\end{proof}

Let us make the following observation. 

\begin{prop}
  Suppose $B$ is a cluster-tilted algebra and $M$ is an indecomposable $B$-module. If $M$ is also an indecomposable $C$-module, for some tilted algebra $C$, whose relation extension is $B$, then $M$ is induced or coinduced.  
  \end{prop}
\begin{proof}
  Observe that if $M$ is a $C$-module, then Proposition \ref{2.2}(b) implies that id$_CM\leq 1$ or pd$_C M \leq 1$.  In the first case $M\cong M\otimes_C B$ and in the second one $M\cong D(B\otimes_C  DM)$ by Proposition \ref{4.3}.  Hence, in either case $M$ is induced or coinduced from a $C$-module.
\end{proof}
   Therefore, it makes sense to ask which indecomposable modules over a cluster-tilted algebra are also indecomposable modules over a tilted algebra.  

\begin{thm}\label{7.1}
Let $B$ be a cluster-tilted algebra.  Then for every transjective indecomposable $B$-module $M$, there exists a tilted algebra $C$, such that $B$ is the relation extension of $C$, and $M$ is an indecomposable $C$-module. In particular, every transjective $B$-module is induced or coinduced from $C$.
\end{thm}

\begin{proof} We will show that if $M$ is a transjective indecomposable $B$-module then $M$ or $\tau^{-1}_BM$ lies in a local slice in $\textup{mod}\,B$. According to Corollary \ref{4.5} this will prove the theorem.

Let $A$ be a hereditary algebra and $T\in\mathcal{C}_A$ a cluster-tilting object such that $B=\text{End}_{\mathcal{C}_A} T$. Let $M$ be an indecomposable $B$-module lying in the transjective component $\mathcal{T}$ of $\Gamma(\text{mod}\,B)$, and let $\widetilde M\in \mathcal{C}_A$ be an indecomposable object  such that $\text{Hom}_{\mathcal{C}_A}(T,\widetilde M)=M$.
Finally, let $\widetilde \Sigma=\Sigma (\rightarrow M) $ be the full subquiver of the Auslander-Reiten of $\mathcal{C}_A$ defined in section~\ref{sect 2.4}.

Since $B\cong \text{End}_{\mathcal{C}_A} (\tau_{\mathcal{C}_A}^\ell T)$, for all $\ell \in \mathbb{Z}$, we may assume without loss of generality that $\widetilde\Sigma$
lies in the preprojective component of $\textup{mod}\,A$.
Furthermore, we may assume that every preprojective successor of $\widetilde\Sigma$ in $\textup{mod}\,A$ is sincere. Indeed, this follows from the fact that there are only finitely many (isoclasses of) indecomposable preprojective $A$-modules that are not sincere. For tame algebras this holds, because non-sincere modules are supported on a Dynkin quiver, and for wild algebras see \cite{Ke}.

Now since $\widetilde M$ is a sincere $A$-module, Proposition \ref{2.6} implies that $\widetilde \Sigma$ is a slice in $\textup{mod}\,A$, and therefore $\widetilde \Sigma $ is a local slice in $\mathcal{C}_A$.
Let $\Sigma=\text{Hom}_{\mathcal{C}_A}(T,\widetilde\Sigma)$. Then $M\in \Sigma$, and
thus, if $\Sigma $ is a local slice in $\textup{mod}\,B$, we are done.

Suppose to the contrary that $\Sigma $ is not a local slice. Then Lemma \ref{lemls} implies that $\widetilde\Sigma$ contains an indecomposable summand $\tau_{\mathcal{C}_A} T_i$ of $\tau_{\mathcal{C}_A} T$.
Let $\widetilde \Sigma '=\Sigma(T_i\rightarrow) $ in $\mathcal{C}_A$. Then $\widetilde \Sigma '$ contains $\tau^{-1}_{\mathcal{C}_A}\widetilde{M}$.
Since $T_i$ is sincere, it again follows from Proposition \ref{2.6} that $\widetilde \Sigma '$ is a local slice in $\mathcal{C}_A$.

Moreover $\widetilde \Sigma '$ cannot contain any summands of $\tau_{\mathcal{C}_A} T$, because if it did, there would be a sectional path from $T_i$ to a summand of $\tau_{\mathcal{C}_A}T$, hence $\text{Hom}_{\mathcal{C}_A}(T_i,\tau_{\mathcal{C}_A}T)\ne 0$, which is impossible, since $T$ is a cluster-tilting object.
Therefore Lemma \ref{lemls} implies that $\text{Hom}_{\mathcal{C}_A}(T,\widetilde \Sigma')$ is a local slice in $\textup{mod}\,B$ containing $\tau_B^{-1}M=\text{Hom}_{\mathcal{C}_A}(T,\tau_{\mathcal{C}_A}^{-1} \widetilde M)$.
\end{proof}

Following \cite{R2}, we say that a cluster-tilted algebra $B$ is {\em cluster-concealed} if $B=\text{End}_{\mathcal{C}_A} (T) $ where $T$ is obtained from a preprojective $A$-module.  This means that all projective $B$-modules lie in the transjective component of $\Gamma(\text{mod}\,B)$.  In this case, we show that the theorem above holds not only for the transjective modules but for all $B$-modules. 

\begin{thm}\label{7.2}
Let $B$ be a cluster-concealed algebra.  Then for every indecomposable $B$-module $M$ there exists a tilted algebra $C$ whose relation extension is $B$, such that $M$ is an indecomposable $C$-module.  Moreover, for all non-transjective modules one can take the same $C$.  In particular, every  $B$-module is induced or coinduced from some tilted algebra.

\end{thm}

\begin{proof}
Observe that by Theorem \ref{7.1} it suffices to consider the case when $M$ does not belong to the transjective component of mod$\,B$.  In this case, $B$ is of infinite representation type.   

Because $B$ is cluster-concealed there exists a hereditary algebra $A$, and  a preprojective tilting $A$-module $T$ such that $B\cong \text{End}_{\mathcal{C}_A}(T) $. Observe that $A$ is of infinite representation type.  Let $C=\text{End}_A (T)$ be the corresponding tilted algebra. Then $B$ is the relation extension of $C$, thus  $B=C\ltimes\text E$, where $E=\text{Ext}^2_C(DC,C)$.  Let $\mathcal{R}_A$ be the set of all regular $A$-modules. It is nonempty because $A$ is of infinite representation type.  It follows from Theorem \ref{2.3}(a) that $\mathcal{R}_A\in\mathcal{T}(T)$.  According to Theorem \ref{2.3}(b) the set of regular $C$-modules $\mathcal{R}_C$ is obtained from $\mathcal{R}_A$ by applying the functor $\text{Hom}_A(T,-)$.  Also, because of Theorem \ref{2.5}, the set of regular $B$-modules $\mathcal{R}_B$ is obtained from $\mathcal{R}_A$ by applying $\text{Hom}_{\mathcal{C}_A}(T,-)$.  Hence, given $M\in\mathcal{R}_B$ there exists $\widetilde{M}\in \mathcal{R}_A$ such that $\text{Hom}_{\mathcal{C}_A}(T,\widetilde{M})\cong M$.  Because $\mathcal{R}_A\in\mathcal{T}(T)$, Theorem \ref{9.1} implies that ${M}\cong \text{Hom}_A(T,\widetilde{M})\otimes _C B$.  Hence, $M$ is induced from the indecomposable $C$-module $\text{Hom}_A(T,\widetilde{M})$, and now it remains to show that $M$ is actually an indecomposable $C$-module itself, that is $M\cong \text{Hom}_A(T,\widetilde{M})$. It follows from Theorem \ref{2.3}(c) that id$_C \text{Hom}_A(T,\widetilde{M})\leq 1$, so Proposition \ref{4.3}(a) implies that 
$$M\cong \text{Hom}_A(T,\widetilde{M})\otimes_C B \cong \text{Hom}_A(T,\widetilde{M}).$$ 
This shows that $M$ is an indecomposable $C$-module.  
\end{proof}

The next result describes a situation when all modules over a cluster-tilted algebra of tame type are induced or coinduced from some tilted algebra.  

\begin{thm}\label{7.3}
Let $B$ be a tame cluster-tilted algebra.  Let $\mathcal{S}_B$ be a tube in {\upshape mod$\,B$} and let 
$$P_B(1), P_B(2), \dots , P_B(\ell)$$
denote all distinct indecomposable projective $B$-modules belonging to $\mathcal{S}_B$.  If 
{\upshape $$\text{Hom}_B(P_B(i),P_B(j))=0 \text{ for all }i\not = j \text{ and } 1\leq i,j\leq \ell$$}then every module in $\mathcal{S}_B$ is induced or coinduced from the same tilted algebra $C$. 
\end{thm}

\begin{proof}
Let $A$ be a hereditary algebra and $T$ a tilting $A$-module such that $B=\textup{End}_{\mathcal{C}_A} T$. Thus $A$ is tame, and we can suppose without loss of generality that $T$ has no preinjective summands. The regular components of $\textup{mod}\,A$ form a family of pairwise orthogonal tubes. 
Let $\mathcal{S}_A$ be the tube in $\textup{mod}\,A$ whose image under $\text{Hom}_{\mathcal{C}_A}(T,-)$ is the tube $\mathcal{S}_B$ in the statement of the theorem. Let $r$ denote the rank of the tube $\mathcal{S}_A$. Let $\mathcal{S}_C$ denote  the connected component in $\textup{mod}\,C$ given as the image of $\mathcal{S}_A$ under the functor $\text{Hom}_A(T,-)$. 
Thus $\mathcal{S}_C$ lies inside $\mathcal{Y}(T)$, which implies that each module in $\mathcal{S}_C$ has projective dimension at most one, by Proposition \ref{2.2}, and therefore the coinduction functor is the identity on $\mathcal{S}_C$, by Proposition \ref{4.3}.

 Theorem \ref{9.1} yields the following commutative diagram.

{\upshape $$\xymatrix@C100pt@R25pt{\mathcal{S}_A\cap\mathcal{T}(T)\ar[r]^{\text{Hom}_A(T,-)}\ar[rd]_{\text{Hom}_{\mathcal{C}_A}(T,-)}&\mathcal{S}_C\ar[d]^{-\otimes _C B}\\ &\quad\mathcal{S}_C\otimes_C B \subset \mathcal{S}_B.}$$}

Let $T_1,T_2,\ldots,T_\ell$ be the indecomposable summands of $T$ that lie inside $\mathcal{S}_A$, such that $P_B(i)=\text{Hom}_{\mathcal{C}_A}(T,T_i)$, for $i=1,2,\ldots \ell$, are the indecomposable projective modules in $\mathcal{S}_B$. Because of our assumption $\text{Hom}_A(T_i,T_j)=0$, if $i\ne j$, we have that each $T_i$ lies on the mouth of $\mathcal{S}_A$, see for example \cite[Proposition XII.2.1]{SimSko3}.

The local configuration  in the Auslander-Reiten quiver of $\mathcal{S}_A$ is the following.

$$\xymatrix@!@C=15pt@R=15pt{\cdots\ar[rd]&&\tau_A^2 T_i\ar[dr]&&\tau_A T_i\ar[rd] &&T_i\ar[dr]&&\cdots\\
&\cdots \ar[ur]\ar[dr]&&\bullet\ar[ru] \ar[dr]&&\bullet\ar[ur]\ar[dr]&&\cdots\ar[ru]\\
&&\cdots \ar[ur]&&\bullet 
\ar[ur]&&\cdots\ar[ru]}$$

The corresponding local configuration  in the Auslander-Reiten quiver of $\mathcal{S}_B$ is the following.

$$\xymatrix@!@C=-30pt@R=-30pt{\cdots\ar[rd]&&I_B(i)\ar[dr]&&&&P_B(i)\ar[dr]&&\cdots\\
&\cdots \ar[ur]\ar[dr]&&\tau_B\,\text{rad}\,P_B(i) \ar[dr]&&\text{rad}\,P_B(i)\ar[ur]\ar[dr]&&\cdots\ar[ru]\\
&&\cdots \ar[ur]&&\bullet 
\ar[ur]&&\cdots\ar[ru]}$$

Since $\text{Ext}^1_A(T_i,-)\cong D\text{Hom}_A(-,\tau T_i)$, and $\tau T_i$ lies on the mouth of $\mathcal{S}_A$, 
we see that ${\mathcal{S}_A\cap \mathcal{T}(T)}$ consists of the $r-\ell$ corays not ending in one of the $\tau T_i$, $i=1,2,\ldots,\ell$.  Therefore $\mathcal{S}_C\otimes_C B$ has $r-\ell $ corays, because of our commutative diagram. So these $r-\ell$ corays are in the image of the induction functor $-\otimes_C B$.

Moreover the $\ell$ corays in $\mathcal{S}_B$ that are not in the image of the induction functor are precisely those ending in $\tau_B\,\textup{rad}\,P_B(i)$, $i=1,2,\ldots,\ell.$ We will show that these corays are equal to the corays in $\mathcal{S}_C$ ending in $\tau_C\,\text{rad} \,P_C(i)$, which implies that these corays are coinduced.

Since $I_B(i)=\text{Hom}_{\mathcal{C}_A}(T,\tau^2_AT_i)$ we have 
\begin{equation}
 \label{eq1}
 \tau_B\,\text{rad}\, P_B(i) = I_B(i)/S(i).
\end{equation}
By our assumption on $T$, we have $\text{Hom}_A(T_i,T)=\text{Hom}_A(T_i,T_i)$ and thus $I_C(i)=S(i)$ is simple.
Moreover, by Proposition \ref{3.6}(b), there is a short exact sequence in $\textup{mod}\,B$ of the form
\[ \xymatrix{0\ar[r]&I_C(i)\ar[r]&I_B(i)\ar[r]&\tau_C\Omega_CI_C(i)\ar[r]&0,}\]
and thus

\begin{equation}
 \label{eq2}
 I_B(i)/S(i)\cong \tau_C\Omega_CI_C(i).
\end{equation}
Again using that $I_C(i)$ is simple, we see that 

\begin{equation}
 \label{eq3}
 \Omega_CI_C(i) =\text{rad}\,P_C(i).
\end{equation}
Combining equations (\ref{eq1})-(\ref{eq3}) yields

\[\tau_B\,\text{rad}\, P_B(i) =  \tau_C \, \text{rad} \,P_C(i).\]

Finally, let $M$ be any indecomposable $C$-module on the coray in $\mathcal{S}_C$ ending in $P_C(i)$, but $M\ne P_C(i)$.
Then $M\otimes_C B$ lies on the coray in $\mathcal{S}_B$ ending in $P_B(i)$. Lemma \ref{3.8} implies that 
\[\tau_B (M\otimes_C B)= D(B\otimes_C D(\tau_C M)) =\tau_C M,\]
where the last identity holds because the coinduction functor is the identity on the tube $\mathcal{S}_C$.
This shows that the $\tau_B$-translate of the coray in $\mathcal{S}_B$ ending in $P_B(i)$ is equal to the coray in $\mathcal{S}_C$ ending in $\tau_C\,\text{rad}\,P_C(i)$.
In particular the modules on this coray are coinduced from $C$.

  This finishes the proof of the theorem. 
\end{proof} 

\begin{remark}
 There are cluster-tilted algebras $B$ with indecomposable modules that are not induced and not coinduced from any tilted algebra.  We give an example below. 
\end{remark}

\begin{exmp}\label{ex 7.7}
Let $B$ be a cluster-tilted algebra of type $\widetilde{\mathbb{A}}_{(4,1)}$ given by the following quiver with relations. 
$$\xymatrix{1\ar@<.5ex>[rr]\ar@<-.5ex>[rr]_{\beta}&&5\ar[dl]^{\epsilon}&&\alpha\beta=\beta\epsilon=\epsilon\alpha=0\\&3\ar[ul]^{\alpha}\ar[dr]^{\delta}&&&\gamma\delta=\delta\sigma=\sigma\gamma=0.\\2\ar[ur]^{\gamma}&&4\ar[ll]_{\sigma}}$$

Then there are exactly two tilted algebras $C$ and $C'$, both of infinite representation type, whose relation extension is $B$.  $C$ and $C'$ can be represented by the following quivers with relations. 

$$\xymatrix{
&1\ar@<.5ex>[rr]\ar@<-.5ex>[rr]_{\beta}&&5&\alpha\beta=0&&&
1\ar@<.5ex>[rr]\ar@<-.5ex>[rr]_{\beta}&&5\ar[dl]^{\epsilon}&\beta\epsilon=0\\
C : &&3\ar[dr]^{\delta}\ar[ul]^{\alpha}&&\gamma\delta=0.&&C': &
&3\ar[dr]^{\delta}&&\gamma\delta=0.\\
&2\ar[ur]^{\gamma}&&4&&&&
2\ar[ur]^{\gamma}&&4}$$

There is a tube in $\Gamma(\text{mod}\,B)$ of rank four containing three distinct projective modules. We describe it below and identify the modules that have the same label.  We will show that the four modules emphasized in bold are not induced and not coinduced. 
$$\xymatrix@!@C=-2pt@R=-2pt{{\begin{smallmatrix}4\\2\end{smallmatrix}}\ar[dr]&&&&
{\begin{smallmatrix}2\\3\\1\\5\\3\\4\end{smallmatrix}}\ar[dr]&&&&
{\begin{smallmatrix}4\\2\end{smallmatrix}}\\&
{\begin{smallmatrix}4\end{smallmatrix}}\ar[dr]&&
{\begin{smallmatrix}3\\1\\5\\3\\4\end{smallmatrix}}\ar[dr]\ar[ur]&&
{\begin{smallmatrix}2\\3\\1\\5\\3\end{smallmatrix}}\ar[dr]&&
{\begin{smallmatrix}2\end{smallmatrix}}\ar[ur]&&\\&&
{\begin{smallmatrix}3\\4\;1\\\;\;\;5\\\;\;\;3\\\;\;\;4\end{smallmatrix}}\ar[dr]\ar[ur]&&
{\begin{smallmatrix}\bf{3}\\\bf{1}\\\bf{5}\\\bf{3}\end{smallmatrix}}\ar[dr]\ar[ur]&&
{\begin{smallmatrix}2\;\;\;\\3\;\;\;\\1\;\;\;\\5\;2\\3\end{smallmatrix}}\ar[dr]\ar[ur]\\&
{\begin{smallmatrix}1\\5\\3\\4\end{smallmatrix}}\ar[dr]\ar[ur]&&
{\begin{smallmatrix}\bf{3}\\\bf{4}\;\bf{1}\\\;\;\;\bf{5}\\\;\;\;\bf{3}\end{smallmatrix}}\ar[dr]\ar[ur]&&
{\begin{smallmatrix}\bf{3}\;\;\;\\\bf{1}\;\;\;\\\bf{5}\;\bf{2}\\\bf{3}\end{smallmatrix}}\ar[dr]\ar[ur]&&
{\begin{smallmatrix}2\\3\\1\\5\end{smallmatrix}}\ar[dr]&&\\
{\begin{smallmatrix}2\;\;\;\\3\;\;\;\\1\;1\\\;5\;5\\\;\;\;\;3\\\;\;\;\;4\end{smallmatrix}}\ar[dr]\ar[ur]&&
{\begin{smallmatrix}1\\5\\3\end{smallmatrix}}\ar[dr]\ar[ur]&&
{\begin{smallmatrix}\bf{3}\;\;\;\;\\\bf{4}\;\bf{1}\;\;\\\;\;\;\;\bf{5}\;\bf{2}\\\;\;\;\;\;\bf{3}\end{smallmatrix}}\ar[dr]\ar[ur]&&
{\begin{smallmatrix}3\\1\\5\end{smallmatrix}}\ar[dr]\ar[ur]&&
{\begin{smallmatrix}2\;\;\;\\3\;\;\;\\1\;1\\\;5\;5\\\;\;\;\;3\\\;\;\;\;4\end{smallmatrix}}\\&
{\begin{smallmatrix}2\;\;\;\\3\;\;\;\\1\;1\\\;5\;5\\\;\;\;\;3\end{smallmatrix}}\ar[ur]\ar[dr]&&
{\begin{smallmatrix}1\;\;\;\\5\;2\\3\end{smallmatrix}}\ar[dr]\ar[ur]&&
{\begin{smallmatrix}3\\4\;1\\\;\;\;5\end{smallmatrix}}\ar[dr]\ar[ur]&&
{\begin{smallmatrix}3\;\;\;\\1\;1\\\;5\;5\\\;\;\;3\\\;\;\;4\end{smallmatrix}}\ar[dr]\ar[ur]\\
\dots\ar[ur]&&\dots\ar[ur]&&\dots\ar[ur]&&
\dots\ar[ur]&&
\dots}$$

The corresponding skew tube in $\Gamma(\text{mod}\,C)$ is given below.  Observe that it consists of one coray ending in the projective at vertex 2, labeled $\begin{smallmatrix}2\\3\\1\\5\end{smallmatrix}$, and four rays starting in ${\begin{smallmatrix}2\\3\\1\\5\end{smallmatrix}}$,${\begin{smallmatrix}3\\1\\5\end{smallmatrix}}$,${\begin{smallmatrix} 4 \end{smallmatrix}}$ and ${\begin{smallmatrix}1\\5\end{smallmatrix}}$.
Let $E=\tau^{-1}\Omega^{-1}C$ and $DE=\tau\Omega DC$, then we make the following computation. 
$$\xymatrix{E:  {\begin{smallmatrix}3\\4\end{smallmatrix}}\oplus {\begin{smallmatrix}3\\4\end{smallmatrix}}\oplus {\begin{smallmatrix}3\\4\end{smallmatrix}}\oplus {\begin{smallmatrix}2\end{smallmatrix}}\oplus{\begin{smallmatrix}3\\4\end{smallmatrix}}&&
DE={\begin{smallmatrix}0\end{smallmatrix}}\oplus{\begin{smallmatrix}4\end{smallmatrix}}\oplus{\begin{smallmatrix}2\\3\\1\\5\end{smallmatrix}}\oplus{\begin{smallmatrix}2\\3\\1\\5\end{smallmatrix}}\oplus{\begin{smallmatrix}0\end{smallmatrix}}}$$

$$\xymatrix@!@C=5pt @R=0pt{
&&&&&&&{\begin{smallmatrix}3\\4\;1\\\;\;\;5\end{smallmatrix}}\ar[dr]&&\dots
\\&&&&
{\begin{smallmatrix}2\\3\\1\\5\end{smallmatrix}}\ar[dr]&&
{\begin{smallmatrix}1\\5\end{smallmatrix}}\ar[ur]\ar[dr]&&
{\begin{smallmatrix}\;3\;\;\;\;\\4\;1\;1\\\;\;\;\;5\;5\end{smallmatrix}}\ar[ur]\ar[dr]\\
&{\begin{smallmatrix}4\end{smallmatrix}}\ar[dr]&&
{\begin{smallmatrix}3\\1\\5\end{smallmatrix}}\ar[ur]\ar[dr]&&
{\begin{smallmatrix}2\;\;\;\;\\3\;\;\;\;\\\;1\;1\\\;\;5\;5\end{smallmatrix}}\ar[ur]\ar[dr]&&
{\begin{smallmatrix}1\;1\\\;\;5\;5\end{smallmatrix}}\ar[ur]\ar[dr]&&
\dots\\&&
{\begin{smallmatrix}3\\4\;1\\\;\;\;5\end{smallmatrix}}\ar[ur]\ar[dr]&&
{\begin{smallmatrix}3\;\;\;\\1\;1\\\;\;5\;5\end{smallmatrix}}\ar[ur]\ar[dr]&&
{\begin{smallmatrix}2\;\;\;\;\;\;\\3\;\;\;\;\;\;\\\;1\;1\;1\\\;\;\;5\;5\;5\end{smallmatrix}}\ar[ur]\ar[dr]&&
\dots\ar[ur]\\&
{\begin{smallmatrix}1\\5\end{smallmatrix}}\ar[ur]\ar[dr]&&
{\begin{smallmatrix}\;3\;\;\;\;\\4\;1\;1\\\;\;\;\;5\;5\end{smallmatrix}}\ar[ur]\ar[dr]&&
{\begin{smallmatrix}3\;\;\;\;\;\;\\1\;1\;1\\\;\;5\;5\;5\end{smallmatrix}}\ar[ur]\ar[dr]&&
\dots\ar[ur]\\
{\begin{smallmatrix}2\;\;\;\;\\3\;\;\;\;\\\;1\;1\\\;\;5\;5\end{smallmatrix}}\ar[ur]\ar[dr]&&
{\begin{smallmatrix}1\;1\\\;\;5\;5\end{smallmatrix}}\ar[ur]\ar[dr]&&
{\begin{smallmatrix}\;\;3\;\;\;\;\;\;\\4\;1\;1\;1\\\;\;\;\;\;5\;5\;5\end{smallmatrix}}\ar[ur]\ar[dr]&&
\dots\ar[ur]\\&
\dots\ar[ur]&&
\dots\ar[ur]&&
\dots\ar[ur]
}$$

Let $\mathcal{O}_{C}$ denote the coray in $\Gamma(\text{mod}\,C)$ ending in the projective at vertex 2.  Observe that if we induce the modules in $\mathcal{O}_C$, that is apply $D\text{Hom}_C(-,DB)$, then we obtain a coray in the tube of $\Gamma(\text{mod}\,B)$ that ends in the projective $B$-module at vertex 2.  Recall that there is an isomorphism of $C$-modules $D\text{Hom}_C(M,DB)\cong M\oplus D\text{Hom}_C(M,DE)$ for every $M\in\text{mod}\,C$.  On the other hand coinduction $\text{Hom}_C(B,-)$ of modules in $\mathcal{O}_C$ will act as the identity map, because there is a $C$-module isomorphism $\text{Hom}_C(B,M)\cong M\oplus \text{Hom}_C(E,M)$ and the last summand is zero for all $M\in\mathcal{O}_{C}$. Therefore, we see the coray $\mathcal{O}_{C}$ appearing in $\Gamma(\text{mod}\,B)$.  Finally, the only module in the tube above that does not belong to $\mathcal{O}_{C}$ is the simple at 4.  Observe, that this module is also a projective $C$-module, so inducing it we obtain the corresponding projective $B$-module $\begin{smallmatrix}4\\2 \end{smallmatrix}$.

Similarly, in $\Gamma(\text{mod}\,C')$ there is a skew tube consisting of four corays and one ray starting in ${\begin{smallmatrix}1\\5\\3\\4\end{smallmatrix}}$, the injective $C'$-module at vertex 4.  Denote this ray by $\mathcal{R}_{C'}$.  Analogous calculations yield that the ray in $\Gamma(\text{mod}\,B)$ starting in ${\begin{smallmatrix}2\\3\\1\\5\\3\\4\end{smallmatrix}}$ the injective $B$-module, is coinduced from $\mathcal{R}_{C'}$.  But the induction of $\mathcal{R}_{C'}$ acts as the identity map, so we see this exact same ray appearing in $\Gamma(\text{mod}\,B)$.  Observe that every module on this ray is indeed an indecomposable $C'$-module.  Finally, the simple injective $C'$-module at vertex 2, belongs to this skew tube, and its induction is again the simple module supported at vertex 2. 

In particular we observe that the $B$-modules which make up the mesh emphasized in bold are not induced and not coinduced from any tilted algebra.  Moreover, as one goes down the tube in $\Gamma(\text{mod}\,B)$ there will appear infinitely many such meshes, where no module is induced or coinduced.  
\end{exmp}
%
%

\bibliographystyle{plain}

\end{document}